\documentclass[12pt,a4paper]{amsart}

\usepackage[latin1]{inputenc}

\usepackage{enumerate}
\usepackage{mathtools}

\usepackage{float}
\restylefloat{table}

\usepackage{graphicx}

\usepackage{hyperref}

\usepackage{amsmath}

\usepackage{amssymb}

\usepackage{amsthm}

\usepackage{tikz-cd}

\usepackage{a4wide}

\newtheoremstyle{plain}
  {6pt}   
  {6pt}   
  {\itshape}  
  {0pt}       
  {\bfseries} 
  {.}         
  {5pt plus 1pt minus 1pt} 
  {}          

\newtheoremstyle{definition}
  {6pt}   
  {6pt}   
  {\normalfont}  
  {0pt}       
  {\bfseries} 
  {.}         
  {5pt plus 1pt minus 1pt} 
  {}          

\theoremstyle{plain}
\newtheorem*{thm*}{Theorem}
\newtheorem{thm}{Theorem}[section]
\newtheorem{prop}[thm]{Proposition}

\newtheorem{lem}[thm]{Lemma}

\theoremstyle{definition}
\newtheorem{defn}[thm]{Definition}

\newtheorem{ex}[thm]{Example}
\newtheorem{rmk}[thm]{Remark}

\numberwithin{equation}{thm}

\newcommand{\emphbf}[1]{\emph{\textbf{#1}}}

\DeclareMathAlphabet{\mathpzc}{OT1}{pzc}{m}{it}

\DeclareMathOperator{\radoperator}{rad}

\DeclareMathOperator{\id}{id}
\DeclareMathOperator{\Hom}{Hom}
\DeclareMathOperator{\Ext}{Ext}

\DeclareMathOperator{\LaMod}{\Lambda-Mod}

\DeclareMathOperator{\CCoMod}{\mathcal{C}-CoMod}
\DeclareMathOperator{\CContraMod}{\mathcal{C}-ContraMod}

\DeclareMathOperator{\CoModC}{CoMod-\mathcal{C}}

\DeclareMathOperator{\End}{End}

\DeclareMathOperator{\res}{res}

\DeclareMathOperator{\Ind}{Ind}

\DeclareMathOperator{\op}{op}

\DeclareMathOperator{\rest}{res}
\DeclareMathOperator{\incl}{inc}

\DeclareMathOperator{\gldim}{gldim}
\DeclareMathOperator{\Triv}{Triv}
\DeclareMathOperator{\twmod}{twmod}

\DeclareMathOperator{\Forget}{Forget}

\newcommand{\Mod}[1]{#1\hbox{-}\mathrm{Mod}}
\newcommand{\modu}[1]{#1\hbox{-}\mathrm{mod}}

\begin{document}

\title[Comparing self-dual corings, Frobenius corings and Ringel self-duality]{Comparing self-dual corings, Frobenius corings and Ringel self-duality}
\author{Tomasz Brzezi\'nski}
\author{Teresa Conde}
\author{Steffen Koenig} 
\author{Julian K\"ulshammer}

\address{Tomasz Brzezi\'nski\\
Department of Mathematics, Swansea University\\
Swansea University Bay Campus\\
Fabian Way\\
  Swansea SA1 8EN, U.K.\ \newline \indent
Faculty of Mathematics, University of Bia{\l}ystok\\ K.\ Cio{\l}kowskiego  1M\\
15-245 Bia\-{\l}ys\-tok, Poland}
\email{T.Brzezinski@swansea.ac.uk}

\address{Teresa Conde\\
  Faculty of Mathematics and TRR 358,
  University of Bielefeld \\
  Universit\"atsstra{\ss}e 25, 33615 Bielefeld, Germany } 
  \email{tconde@math.uni-bielefeld.de}
  
\address{Steffen Koenig\\
Institute of Algebra and Number Theory,
University of Stuttgart \\ Pfaffenwaldring 57 \\ 70569 Stuttgart,
Germany} \email{skoenig@mathematik.uni-stuttgart.de}

\address{Julian K\"ulshammer\\
Department of Mathematics, Uppsala University\\
Box 480 \\ 75106 Uppsala, Sweden} \email{julian.kuelshammer@math.uu.se}


\begin{abstract}
    Self-dual corings and self-dual algebra extensions are
    classified and these (self-)dualities are compared with Ringel
    (self-)duality, revealing fundamental differences.
    To each coring $\mathcal C$, two algebras are associated, known as the left
    and the right dual algebra of $\mathcal C$.
    It is shown that these two algebras coincide in a natural way if and
    only if $\mathcal C$ is a Frobenius coring. Various other equivalent
    characterisations of $\mathcal C$ being self-dual
    are given, in terms of certain
    algebra extensions being Frobenius extensions and in terms of certain
    forgetful or restriction functors being Frobenius functors. Ringel self-duality however is shown to be rather different, for which a homological explanation is given.
    \end{abstract}

\maketitle

\section{Introduction and main results}

Dualities in algebra often relate algebraic objects to interesting
companions, such as relating an algebra with its opposite algebra. Self-duality then exhibits a particular symmetry.
Certain dualities do, however, start with a particular object and assign
to it a pair of objects each with different structure from the original object,
but dual to each other with respect to their structure. 

In this article, we study corings \cite{Swe75}, that is, comonoids in the monoidal category of bimodules over a given algebra, a far-reaching generalisation of coalgebras. These occur in a variety
of contexts, for instance of classical 
or weak Hopf algebras \cite{BNS99} and
particular phenomena captured by Galois corings \cite{Brz05}, such as principal fibre bundles in noncommutative geometry \cite{BM93}. Also algebra extensions are described
by their canonical (or Sweedler) corings. In representation theory, corings occur as
bocses \cite{Roi79}, differential tensor categories \cite{BSZ09} and curved differential graded algebras \cite{Brz13}.

In the context of corings, an instance of duality is given by assigning to a coring its left and right dual algebras. Self-duality of the coring then means that the left and the
right dual algebra coincide in a natural way.
 For bocses, Burt--Butler duality is another
duality operation, assigning to each bocs a left and a right algebra. We 
will see that, up to convention, Burt--Butler
duality  coincides with coring duality.

Burt--Butler duality is in turn related to Ringel duality, a
crucial feature of quasi-hereditary algebras and of highest weight
categories. However, as we shall see, this connection does not identify coring self-duality with Ringel self-duality. We characterise their intersection and provide a homological explanation for the distinction between the two notions of self-duality.

\subsection*{The first main result}

We characterise and classify corings that are self-dual in the sense described above. It turns out that these are exactly a particular kind of
corings that have been studied for different reasons -- the Frobenius
corings (\cite{Brz03}) -- which by definition are equipped with a rather different kind
of duality in ring theoretic terms.

The category of Frobenius corings is equivalent to the category of
Frobenius extensions introduced in \cite{Kas54} and, independently, in \cite{NT59}, and indeed self-duality in the co\-ring sense turns out
to be self-duality of corresponding algebra extensions. In other words,
Frobenius ring extensions are exactly the self-dual algebra
extensions.

A third kind of Frobenius data are Frobenius functors (introduced in \cite{Mor65} as strongly adjoint pairs), which by definition
have a left adjoint that is also a right adjoint. They appear in our context in the form of
certain forgetful or restriction functors being Frobenius functors.

The first main theorem of this text shows by a sequence of eleven
mutually equivalent statements
that, for an algebra $A$, an $A$-coring $\mathcal{C}=(A,V)$ is a Frobenius coring
(assertion \eqref{mainthm1:1}) if and only if it satisfies one of three equivalent kinds
of dualities (assertions \eqref{mainthm1:2}, \eqref{mainthm1:3} and \eqref{mainthm1:4}) if and only if a restriction/induction/coinduction
functor is a Frobenius
functor (assertions \eqref{mainthm1:5}, \eqref{mainthm1:6} and \eqref{mainthm1:7}) if and only if one of two algebra
extensions involving $A$ and the left/right Burt--Butler algebra of
$\mathcal{C}=(A,V)$ is a Frobenius extension (assertions \eqref{mainthm1:8} and \eqref{mainthm1:9}) if and only if
$V$ as a bimodule is projective and finitely generated on one side and isomorphic to its dual
(assertions \eqref{mainthm1:10} and \eqref{mainthm1:11}).  

{\setcounter{thm}{0}

\renewcommand*{\thethm}{\Alph{thm}}
\begin{thm} \label{mainthm1}
  Let $A$ be an algebra and $\mathcal{C} = (A,V)$ be an $A$-coring with left and
  right dual algebras ${}^{\ast}{\mathcal C} =  \Hom_A({}_AV,{}_AA)$ and
  ${\mathcal C}^{\ast} =  \Hom_{A}(V_A,A_A)$ and left and
  right Burt--Butler algebras $L$ and $R$. Denote the canonical inclusions
  of algebras by $\incl\colon A^{\op} \rightarrow {}^{\ast}{\mathcal C}$, $\incl\colon A^{\op}
  \rightarrow {\mathcal C}^{\ast}$, $\incl\colon A \rightarrow L$ and
  $\incl\colon A \rightarrow R$, respectively.
  
  Then the following assertions are equivalent:
\begin{enumerate}
\item\label{mainthm1:1} The coring $\mathcal{C}$ is a Frobenius $A$-coring.

\item\label{mainthm1:2} The coring $\mathcal{C}$ is self-dual in the following sense. The module ${}_AV$ is finitely generated projective and
  there is an algebra isomorphism $\delta\colon \mathcal{C}^\ast \rightarrow {}^{\ast}\mathcal{C}$
  between its right and its left dual
  ring, restricting to the identity on the common subalgebra $A^{\op}$,
  i.e.\ there is a commutative diagram
  
   \begin{center}
  \begin{tikzcd}
  \mathcal{C}^{\ast}\ar{rr}{\cong}[swap]{\delta} &    &  {}^{\ast}\mathcal{C}\\
  &  \arrow{ul}{\incl} A^{\op}. \arrow{ur}[swap]{\incl} & 
\end{tikzcd}
\end{center}
  
\item\label{mainthm1:3} The coring $\mathcal{C}$ is Burt--Butler self-dual in the following sense. The module ${}_AV$ is finitely generated projective and
  there is an algebra isomorphism $\delta\colon R \rightarrow L$
  between its right and its left
  algebra, restricting to the identity on the common subalgebra $A$,
  i.e.\ there is a commutative diagram 

   \begin{center}
  \begin{tikzcd}
  R  \ar{rr}{\cong}[swap]{\delta}&& L  \\
  &  \arrow{ul}{\incl} A. \arrow{ur}[swap]{\incl} & 
\end{tikzcd}
\end{center}

\item\label{mainthm1:4} The coring $\mathcal{C}$ is Burt--Butler self-dual in a natural way. That is, the module ${}_AV$ is finitely generated projective and there is
  an isomorphism $\delta\colon R \rightarrow L$ of right and left Burt--Butler algebras
  fixing the common subalgebra $A$ 
\begin{center}
  \begin{tikzcd}
  R  \ar{rr}{\cong}[swap]{\delta} && L  \\
  &  \arrow{ul}{\incl} A \arrow{ur}[swap]{\incl} & 
\end{tikzcd}
\end{center}
such that the induced Morita equivalence $F\colon \Mod{R} \rightarrow \Mod{L}$
  makes the following diagrams commutative:

{\footnotesize
\begin{tikzcd}
  \Mod{R} \ar{rr}{\simeq}[swap]{F} && \Mod{L}  \\
  &  \arrow{ul}{R \otimes_A -} \Mod{A} \arrow{ur}[swap]{\Hom_A(L,-)} & 
\end{tikzcd}
\mbox{~\;and\;\;}
\begin{tikzcd}
  \Mod{R} \arrow{dr}[swap]{\rest} \ar{rr}{\simeq}[swap]{F} && \Mod{L} \arrow{dl}{\rest} \\
  &  \Mod{A}. &   
\end{tikzcd}
}

\item\label{mainthm1:5} Induction $V \otimes_A -\colon \Mod{A} \rightarrow \CCoMod$ is a
  Frobenius functor and ${}_AV$ is a finitely generated projective module.

\item\label{mainthm1:6} Coinduction $\Hom_A({}_AV,-)\colon \Mod{A} \rightarrow \CContraMod$ is a Frobenius
  functor and $V_A$ is a finitely generated projective module.

\item\label{mainthm1:7} The forgetful functor ${}_A\res\colon
  \CCoMod \rightarrow \Mod{A}$ is a Frobenius functor.  

\item\label{mainthm1:8} The algebra extension $A \subseteq R$ is a Frobenius extension and
  ${}_A V$ is a finitely generated projective module.

\item\label{mainthm1:9}  The algebra extension $A \subseteq L$ is a Frobenius extension and
  $V_A$ is a finitely generated projective module. 

\item\label{mainthm1:10} There is an isomorphism of $A$-$R$-bimodules $V \cong R$ and
  ${}_A V$ is a finitely generated projective module.

\item\label{mainthm1:11} There is an isomorphism of $L$-$A$-bimodules $V \cong L$ and
  $V_A$ is a finitely generated projective module.
\end{enumerate}  
\end{thm}

All the notions that appear in Theorem~\ref{mainthm1} are recalled in Section~\ref{sec:proofA}, where also the proof of the theorem is given.
We take condition \eqref{mainthm1:2} as definition of coring self-duality (see \cite{BW03}),
condition \eqref{mainthm1:3} as definition of Burt--Butler self-duality (see \cite{BB91}),
and condition \eqref{mainthm1:4} as definition
of natural Burt--Butler self-duality. Hence, the theorem shows that these
three definitions of self-duality are equivalent.  In addition, in Section~\ref{sec:nec_cond}, we provide examples of trivial extensions and coextensions, proving that the condition of $V_A$ being finitely generated projective is necessary.

Theorem \ref{mainthm1}
suggests to compare coring (self-)duality and
Burt--Butler (self-)duality with a kind of algebra extensions coming up
in the following applications. Quasi-hereditary algebras (see \cite{Sco87, CPS88, DR92})
arise frequently
in algebraic Lie theory, for instance as Schur algebras of reductive
algebraic groups or as blocks of
the Bernstein--Gelfand--Gelfand category $\mathcal O$ of a semisimple complex
Lie algebra, and also in representation theory of finite-dimensional
algebras, for instance as algebras of global dimension at most two.

In \cite{KKO14} it has been shown that every quasi-hereditary algebra is Morita
equivalent to one that has a regular
exact Borel subalgebra, which has properties
similar to those of Borel subalgebras in Lie theory. Then the
inclusion of the exact Borel subalgebra $A$ into the quasi-hereditary
algebra $R$ defines an algebra extension, or equivalently a coring (see
\cite{KKO14} and \cite{BKK20} for more information and background).
Applying Burt--Butler duality to this coring relates the right
algebra (which is Morita equivalent to $R$) to the left algebra,
that is quasi-hereditary as well and
known as the Ringel dual of $R$. Ringel self-duality of $R$
is a refinement of the existence of an isomorphism (or a Morita
equivalence) between the Burt--Butler left and right algebras.
Although such self-duality is generally rare, it does occur in several important settings, including the Bernstein--Gelfand--Gelfand (BGG) category $\mathcal{O}$ (\cite{Soe90}), classical and quantised Schur algebras (\cite{Don93,Don92}) and parabolic categories 
$\mathcal{O}$ (\cite{CM17}); see also \cite{EH02,CE18}.

Ringel duality, and in particular Ringel self-duality, has become a strong tool
with applications in representation theory, invariant theory and Lie theory.
While Ringel self-duality is known to hold (or not to hold) for interesting
classes of examples, there are no general criteria for Ringel self-duality
and there is a lack of insight in structural properties or symmetries of
this setup. The connection, but also the contrast between coring/Burt--Butler duality and Ringel duality raises the question of characterising the intersection of Frobenius corings with Ringel self-dual algebra extensions.

\subsection*{The second main result} We address this problem in a more general context
involving regular exact subalgebras of arbitrary finite-dimensional algebras over a field.
The approach is based on relative homological algebra. The notions that
appear in Theorem \ref{mainthm3} and results needed from relative
homological algebra are recalled in Section \ref{sec:proofB},
where also the proof of the
theorem is given.  Statement \eqref{mainthm3:c} in particular solves the problem  for the important subclass of quasi-hereditary algebras. Half of the statement follows from the more general, but more technical parts \eqref{mainthm3:a} and \eqref{mainthm3:b}. However, we also give an independent easier proof. 

\renewcommand*{\thethm}{\Alph{thm}}
\begin{thm} \label{mainthm3}
\begin{enumerate}[(a)]
\item\label{mainthm3:a} Let $A\subseteq R$ be an extension of finite-dimensional algebras and
  suppose $A$ is a regular exact subalgebra such that the category  $\Ind_A^R$ of induced modules is a
  resolving subcategory and the relative global dimension of $R$ with respect to $A$ is finite. Suppose further
  that $A\subseteq R$ is a Frobenius extension. 
  Then $R$ is relative semisimple in the sense that every $R$-module $M$ is
  relative projective. In fact, $M$ is isomorphic to a module of the form
  $R\otimes_A N$ for some $A$-module $N$. 
\item\label{mainthm3:b}  Keep the assumptions of (a) and assume in addition that $R\otimes_A -$ 
  reflects isomorphisms between simple $A$-modules and that 
  $\End_A(L_\mathtt{i})\cong\End_R(R\otimes_A L_\mathtt{i})$ for all $\mathtt{i}$. Then
  $A=R$. 

\item\label{mainthm3:c} Let $R$ be a right standardly stratified algebra and let $A\subseteq R$ be a
  regular exact Borel subalgebra. Suppose further that $A\subseteq R$ is a
  Frobenius extension. Then $A=R$ and this algebra is semisimple.
  \end{enumerate}
\end{thm}

Theorem~\ref{mainthm3} and its proof also give a homological
reason why Ringel duality
of algebras behaves differently from Burt--Butler duality of algebra  extensions.

\medskip
}

Another result on self-duality for a different class of bocses has been
obtained by Andrew Kharchuk in \cite{Kha03}. There, Drozd bocses of
directed algebras are studied. These bocses play a crucial role in connecting
module categories with representations of bocses, in particular in the proof of
Drozd's Tame and Wild Theorem. Drozd bocses of directed algebras correspond to
certain quasi-hereditary algebras which are shown to be Ringel self-dual
if and only if the
original algebra is hereditary. There does not seem to be a direct connection
between Kharchuk's result and Theorem \ref{mainthm3}.

\bigskip

Theorem~\ref{mainthm1} will be shown in the next section, Theorem~\ref{mainthm3} will
be proven in the final section.

\subsection*{Conventions}  All algebras are assumed to be associative and with identity over a (fixed) commutative ring. All subalgebras are assumed to be unital. Given an algebra $A$ and a right $A$-module $M$, we often do not indicate explicitly canonical isomorphisms such as $M\cong M\otimes_AA$ or $\Hom_A(A_A,M_A)\cong M$; they are implicitly understood as present whenever needed. Similar conventions are used for left modules. The category of left (resp.\ right) $A$-modules is denoted by $\Mod{A}$ (resp.\ $\mathrm{Mod}$-$A$), while the category of finitely generated left $A$-modules is denoted by $\modu{A}$. For all categorical notions not explained in this text we refer the reader to \cite{McL98}.

\section{Proof of Theorem~\ref{mainthm1}}\label{sec:proofA}

\subsection{Duality of corings and Burt--Butler duality}
We start by recalling main notions needed for Theorem~\ref{mainthm1}.
\begin{defn}
  Let $A$ be an algebra. An \emphbf{$A$-coring} $\mathcal{C} = (A,V)$
  is an $A$-$A$-bimodule $V$ with
  a $A$-$A$-bimodule homomorphism, called comultiplication,   
  $\mu \colon V\to V\otimes_A V$, which is coassociative, and a morphism of $A$-$A$-bimodules, called counit, 
  $\varepsilon\colon V\to A$, i.e.\ the following diagrams commute:
\[\begin{tikzcd}
V\arrow{r}{\mu}\arrow{d}{\mu} &V\otimes_A V\arrow{d}{\mu\otimes \id_V}\\
V\otimes_A V\arrow{r}{\id_V\otimes \mu} &V\otimes_AV\otimes_A V
\end{tikzcd} \text{ and } 
\begin{tikzcd}
  V\otimes_A A\arrow{dr}[swap]{\cong} &V\otimes_A V\arrow{l}[swap]{\id_V\otimes
    \varepsilon} \arrow{r}{\varepsilon\otimes \id_V} &
  A\otimes_A V\arrow{dl}{\cong}\\ &V.\arrow{u}{\mu}
\end{tikzcd}
\]
The pair $(A,V)$ is called a \emphbf{bocs}.
\end{defn}

A coring $\mathcal C$ has both a left and a right dual algebra
(see \cite{Swe75} or \cite[17.8]{BW03}). 
 As explained (and introduced) in \cite{Swe75}, given a coring $\mathcal{C}=(A,V)$, the right dual bimodule ${V}^{\ast}=
\Hom_{A}(V_A,A_A)$ admits an associative multiplication 
\begin{equation}\label{rd.mult}
    f \ast^r g\colon \begin{tikzcd}
    V \arrow{r}{\mu} & V \otimes_A V \arrow{r}{f \otimes \id_V} &
A \otimes_A V \arrow{r}{\simeq} & V \arrow{r}{g} & A\,.
\end{tikzcd}
\end{equation} 
Symmetrically, the left dual bimodule ${}^{\ast}{V} = \Hom_A({}_AV,{}_AA)$ can be equipped with an associative multiplication
\begin{equation}\label{ld.mult}
    f \ast^l g\colon \begin{tikzcd}
    V \arrow{r}{\mu} & V \otimes_A V \arrow{r}{\id_V \otimes g} &
V \otimes_A A \arrow{r}{\simeq} & V \arrow{r}{f} & A\,.
\end{tikzcd}
\end{equation}  
In both cases, the identity is given by the counit $\varepsilon$.

\begin{defn}\label{def.dual}
  Let $\mathcal{C}=(A,V)$ be a coring.
  The right dual bimodule $V^{\ast}$ with multiplication \eqref{rd.mult} and identity $\varepsilon$ is called the \emphbf{right dual algebra}  of $\mathcal{C}$ and is denoted by $ 
{\mathcal C}^{\ast}$.
 The left dual bimodule ${}^{\ast}V$ with multiplication \eqref{ld.mult} and identity $\varepsilon$ is called the \emphbf{left dual algebra}  of $\mathcal{C}$ and is denoted by $ 
{}^{\ast}{\mathcal C}$.
\end{defn}

We note in passing that there are inclusion algebra anti-homomorphisms 
\begin{equation}\label{inclusions}
    \incl\colon A \rightarrow {\mathcal C}^{\ast},\quad 
a \mapsto \varepsilon(a-)\quad  \mbox{and} \quad 
\incl\colon A \rightarrow {}^{\ast}{\mathcal C},\quad  a \mapsto
  \varepsilon(-a).
\end{equation}

In the context of bocses, Burt and Butler \cite{BB91} assigned a left and
a right algebra to a coring. Their purpose was to develop Auslander--Reiten theory for (particular) bocses coming up in the context of Drozd's
Tame and Wild Theorem. The basic idea in \cite{BB91}, and also in the article \cite{BK90} by Bautista and Kleiner on the same topic,
was to realise the category of
representations of a bocs as a category of induced or coinduced modules
inside module categories of some algebras to be constructed from the bocs.
This then allows one to use criteria like those of
Auslander and Smal{\o} \cite{AS81} to prove the existence of almost split sequences. Given
a bocs $(A,V)$, Burt and Butler constructed the left and right
algebras $L$ and $R$ and
induction and coinduction functors, respectively. Moreover, they
proved that the underlying $A$-coring $V$
is a left $L$- and right $R$-cotilting module.

We are following the exposition in
\cite[Section 10]{KKO14}, except that we talk about corings instead of
bocses (called boxes in \cite{KKO14}). 

\begin{defn}\label{def.BB}
  Let $\mathcal{C}=(A,V)$ be a coring. The {\bf\em right Burt--Butler algebra} $R$ of $\mathcal C$ is
  the opposite algebra of the left dual algebra of $\mathcal{C}=(A,V)$, i.e.\ as a vector space it is $\Hom_A({}_AV,{}_AA)$ with multiplication of $f,g \in R$
  \[f \cdot^r g\colon \begin{tikzcd}
    V \arrow{r}{\mu} & V \otimes_A V \arrow{r}{\id_V \otimes f} &
V \otimes_A A \arrow{r}{\simeq} & V \arrow{r}{g} & A\,.
\end{tikzcd}  \]

  The {\bf \em left Burt--Butler algebra} $L$ of $\mathcal C$ is
  the opposite algebra of the right dual algebra of $\mathcal{C}=(A,V)$, i.e.\ as a vector space it is
  $\Hom_A(V_A,A_A)$ with multiplication of $f,g \in L$
  \[f \cdot^l g\colon \begin{tikzcd}
    V \arrow{r}{\mu} & V \otimes_A V \arrow{r}{g \otimes \id_V} &
A \otimes_A V \arrow{r}{\simeq} & V \arrow{r}{f} & A\,.
\end{tikzcd}  \]
\end{defn}

In \cite{KKO14} bocses or corings are used to describe categories
${\mathcal F}(\Delta)$
of modules with standard filtrations over a quasi-hereditary algebra
$\Lambda$.
These categories are usually not abelian, but always exact, and a main
result of \cite{KKO14} (Theorem 8.2) states that ${\mathcal F}(\Delta)$
is equivalent, as an
exact category, to the category of representations of a bocs (or coring)
$\mathcal{C}=(A,V)$. In this context, both the left and the right Burt--Butler algebra 
are quasi-hereditary (\cite[Theorem 11.2]{KKO14}). In fact, the right algebra $R$ is
Morita equivalent to the given quasi-hereditary algebra $\Lambda$ and the
left algebra $L$ is Morita equivalent to the Ringel dual of $\Lambda$.
More about translating results between bocses and corings in this context
can be found in \cite{BKK20}.

That the terms left and right (dual) algebra in Definition~\ref{def.dual} and Definition~\ref{def.BB}  are
placed differently and in the case of left and
right Burt--Butler algebras there are opposite algebras may look confusing, but
it has a pleasant consequence: Both $L$ and $R$ are algebra extensions of $A$
(and not of $A^{\op}$), which suggests to ask when these algebra extensions
are Frobenius extensions. This motivates conditions \eqref{mainthm1:6} and \eqref{mainthm1:7} in
Theorem \ref{mainthm1}.

There is a natural categorical interpretation of the algebras in Definitions~\ref{def.dual} and \ref{def.BB}. Given an $A$-$A$-bimodule $V$, the tensor functor $V\otimes_A -$ defines a coring $\mathcal{C}=(A,V)$ if and only $V\otimes_A -$ is a comonad on the category of left $A$-modules (see e.g.\ \cite{BBW09}). Equivalently, the functor $\Hom_A ({}_AV,-)$ has to be a monad on the category of right $A$-modules.

The structure maps $\mu$, $\varepsilon$ of the coring $\mathcal{C}=(A,V)$ correspond to the structure natural transformations of the associated comonad by  $\mu\otimes -$ and $ \varepsilon\otimes -$ and to $\Hom_A(\mu,-)$ and $\Hom_A(\varepsilon,-)$ of the associated monad. The Eilenberg--Moore category of $V\otimes_A-$ is the familiar category $\CCoMod$  of left $\mathcal{C} = (A,V)$-\emphbf{comodules}, that is, left $A$-modules $M$ with left $A$-module maps $\lambda\colon M\to V\otimes_A M$, known as {\emphbf{coactions}} rendering commutative the following diagrams
\[\begin{tikzcd}
M\arrow{r}{\lambda}\arrow{d}{\lambda} &V\otimes_A M\arrow{d}{\mu\otimes \id_M}\\
V\otimes_A M\arrow{r}{\id_V\otimes \lambda} &V\otimes_AV\otimes_A M
\end{tikzcd} \text{ and } 
\begin{tikzcd}
  A\otimes_A M\arrow{dr}[swap]{\cong} &V\otimes_A M\arrow{l}[swap]{
    \varepsilon \otimes \id_M} \\ &M\arrow{u}{\lambda} .
\end{tikzcd}
\]
Morphisms are left $A$-linear maps preserving the coactions. The Eilenberg--Moore category of $\Hom_A ({}_AV,-)$ is the less familiar category $\CContraMod$ of left {\emphbf {contramodules}} in the case of coalgebras introduced already in \cite{EM65b} (see e.g.\ \cite{BBW09} for a more detailed discussion and \cite{Pos10} for applications in homological algebra). The objects of $\CContraMod$ are left $A$-modules $N$ together with left $A$-linear maps $\nu_N\colon \Hom_A(_AV,_AN)\to N$ satisfying associativity and unitality conditions. 

The module $V$ is an object in $\CCoMod$ with the structure map $\mu$. Endomorphisms $\hat f$ of $V$ as a left $\mathcal{C}$-comodule are in one-to-one correspondence with left linear maps $f\colon V\to A$ by $\hat f = (1\otimes f)\circ \mu$ and $f = \varepsilon\circ \hat f$. Under this identification $\widehat{f *^lg} = \hat f\circ\hat g$, and so the endomorphism algebra of $V$ as a left $\mathcal{C}$-comodule is isomorphic to the left dual algebra of $\mathcal{C}$. 

The Kleisli categories of $V\otimes_A -$ and $\Hom_A ({}_AV,-)$, i.e.\ the categories of induced comodules, resp.\ coinduced contramodules, are mutually isomorphic; see e.g.\ \cite[3.9]{BBW09}. 
By identifying $\Hom_A({}_AA,\Hom_A(_AV,{}_AA))\cong \Hom_A(_AV,{}_AA)$ one finds that the opposite of the  endomorphism algebra of $A$ in the Kleisli category of $\Hom_A ({}_AV,-)$ (equivalently, $V\otimes_A -$) is isomorphic to the  right Burt--Butler algebra $R$ of $\mathcal{C}$.

\subsection{Frobenius extensions and corings}

In this subsection we recall the definition of Frobenius corings \cite{Brz03}
and collect results about these corings that will be used in the proof of
Theorem \ref{mainthm1}. Details and proofs as well as references to the
original literature can be found in \cite[Section 27]{BW03}. The connection
to the widely studied Frobenius extensions introduced by Kasch \cite{Kas54} and by Nakayama and Tsuzuku \cite{NT59} will be explained in Theorem~\ref{mainthm2} (see also e.g.\ \cite{Kad99} for applications of Frobenius extensions to the Jones theory of subfactors for von Neumann algebras, knot theory and topological quantum field theory). 

\begin{defn}
\label{defn:Frobenius_coring}
  A coring $\mathcal{C}=(A,V)$ is called a \emphbf{Frobenius coring} if there exists an
  $A$-$A$-bimodule map $\eta\colon A \rightarrow V$ and a $\mathcal{C}$-$\mathcal{C}$-bicomodule map
  $\pi\colon V \otimes_A V \rightarrow V$ such that 
the following diagram commutes:
\[\begin{tikzcd}
V\arrow{r}{\id_V\otimes \eta}\arrow{d}[swap]{\eta\otimes \id_V}\arrow{rd}[description]{\id_V}&V\otimes_A V\arrow{d}{\pi}\\
V\otimes_A V\arrow{r}[swap]{\pi}&V\,.
\end{tikzcd}\]
\end{defn}
Evaluating $\eta$ at the identity of $A$, one obtains an element $e$ in the centre of $V$, that is,
\[
e \in V^A \coloneq\{v\in V\;|\; \forall\,\, a\in A: va=av\}.
\]
Furthermore, composing the counit of $\mathcal{C}=(A,V)$ with $\pi$ one obtains the $A$-$A$-bimodule map $\sigma\coloneq \varepsilon \circ\pi\colon  V\otimes_AV\to A$.  Together $(\sigma,e)$ are known as a \emphbf{reduced Frobenius system} and they satisfy
\begin{equation}\label{red.Frob}
    (\sigma \otimes \id_V)\circ (\id_V\otimes \mu) = (\id_V\otimes \sigma)\circ (\mu\otimes \id_V), \quad \sigma(v\otimes e)=\sigma(e\otimes v) = \varepsilon(v), 
\end{equation}
for all $v\in V$.

A coring $\mathcal{C}=(A,V)$ is a Frobenius coring, i.e.~it is endowed with two maps $\eta$ and $\pi$ satisfying the conditions in Definition \ref{defn:Frobenius_coring}, if and only if it is endowed with a reduced Frobenius system, that is a pair $(\sigma,e)$ satisfying \eqref{red.Frob}. The first condition in \eqref{red.Frob} reflects the bicomodule map property of $\pi$, while the second one corresponds to the commutativity of the above diagram.

As the assertions in Theorem \ref{mainthm1} use varying assumptions on
$V$ being left or right projective and finitely generated over $A$, the
following result (appropriately called finiteness of Frobenius
corings in \cite[27.9]{BW03}) helps to compare such assumptions:

\begin{prop} \label{frobfinite}
  Let $\mathcal{C}=(A,V)$ be a Frobenius coring. Then $V$ is finitely generated and
  projective both as a right and as a left $A$-module.
\end{prop}

Again following \cite[Section 27]{BW03} we recall the definition and some
properties of Frobenius extensions:

\begin{defn}
  An algebra homomorphism $A \rightarrow \Lambda$ is called an \emphbf{algebra 
    extension}. An algebra extension is called a \emphbf{Frobenius extension}
   (of the first kind) if $\Lambda$ is  finitely generated projective as a
  right $A$-module and $\Lambda$ is isomorphic to $\Hom_A(\Lambda_A,A_A)$ as an
  $A$-$\Lambda$-bimodule.
\end{defn}

Here, the $A$-$\Lambda$-bimodule structure on $\Hom_A(\Lambda_A,A_A)$ is given by
$(afx)(y)=af(xy)$ for $f \in \Hom_A(\Lambda,A)$ and $a \in A$ and
$x,y \in \Lambda$. 

The defining condition of a Frobenius extension is equivalent to each of
the following conditions: 
\begin{itemize}
\item $\Lambda$ is finitely generated projective as a left $A$-module and $\Lambda$
is isomorphic to the $\Lambda$-$A$-bimodule $\Hom_A({}_A\Lambda,{}_AA)$. 
\item There is an $A$-$A$-bimodule map (called \emphbf{Frobenius homomorphism}) $E\colon \Lambda
\rightarrow A$ and an element (called \emphbf{Frobenius element})
$\beta = \sum_i b_i \otimes b^i \in \Lambda \otimes_A
\Lambda$  such that for all $x \in \Lambda$ there
are equalities $\sum_i E(xb_i)b^i = x = \sum_i b_i E(b^i x)$.
\end{itemize}

Frobenius extensions are closely related with Frobenius corings:

\begin{prop}(\cite[27.6]{BW03}) \label{prop.Fro}
  Let $A \rightarrow \Lambda$ be a Frobenius extension with a Frobenius
  homomorphism $E$ and a Frobenius element $\beta$. Then $\Lambda$  is a
  Frobenius $A$-coring with comultiplication 
  $\Lambda \rightarrow \Lambda \otimes_A
  \Lambda$ mapping $\alpha$ to $\alpha \beta$, and a counit $E$.
\end{prop}

 Proposition~\ref{prop.Fro} extends the corresponding relation between Frobenius algebras and coalgebras observed by Abrams in \cite{Abr96,Abr99}.

There is an isomorphism of categories from the category of
Frobenius corings  over $A$, in which objects are Frobenius corings $(\mathcal{C}=(A,V),\pi,\eta)$ with morphisms preserving $\pi$ and $\eta$ (or, equivalently, preserving central elements $e$ corresponding to $\eta$), to the category of Frobenius extensions
over $A$, whose inverse on objects is defined by the above proposition;
see \cite[27.16]{BW03} for details.

Various characterisations of Frobenius corings among general corings are
known. The following result, Theorem \ref{frobchar}, collects those characterising properties given
in 27.8, 27.10, 27.13 an 27.14 in \cite{BW03} that will be used in the
proofs of the main theorems.
The formulation below uses the identification of left and right dual algebras
of a coring with opposite algebras of left and right Burt--Butler algebras. It also uses the equivalence of
categories of Frobenius corings and Frobenius extensions given above.

A \emphbf{Frobenius functor} is a functor that has both a left adjoint and a
right adjoint and the two adjoints coincide. The two functors together
then are called a \emphbf{Frobenius pair} or an \emphbf{ambidextrous adjunction}. Such pairs of functors were introduced by Morita in \cite{Mor65} under the name \emphbf{strongly adjoint pair}; this terminology is superseded by the one used in the present text which originates from \cite{CMZ97}. 

\begin{thm}\label{frobchar}
  Let $\mathcal{C}=(A,V)$ be a coring. Then the following statements are equivalent:
\begin{enumerate}
\renewcommand{\theenumi}{\Alph{enumi}}
  \item\label{frobchar:1} $\mathcal{C}=(A,V)$ is a Frobenius coring.
\item\label{frobchar:2} The forgetful functor $\res_A\colon \CoModC \rightarrow \mathrm{Mod}\hbox{-}A$ is a
  Frobenius functor. 
\item\label{frobchar:3} The forgetful functor ${}_A\res\colon \CCoMod \rightarrow \Mod{A}$ is a
  Frobenius functor.
\item\label{frobchar:4} $V$ is a finitely generated projective left $A$-module and the algebra 
  extension $A \rightarrow R$ is a Frobenius extension.
\item\label{frobchar:5} $V$ is a finitely generated projective right $A$-module and the algebra 
  extension $A \rightarrow L$ is a Frobenius extension.  
\end{enumerate}
\end{thm}

Another connection between Frobenius extensions and Frobenius corings occurs when one starts with 
an algebra extension $A\to \Lambda$ such that
     $\Lambda_A$ (or ${}_A\Lambda$) is a finitely generated projective module. In this case the dual $A$-$A$-bimodule $\Lambda^* = \mathrm{Hom}_A(\Lambda_A,A_A)$ is an  $A$-coring with the counit given by the evaluation of  $f\in \Lambda^*$ at the identity of $\Lambda$ and the comultiplication
\[
\mu(f) = \sum_i f(e_i-)\otimes f^i,
\]
where $\{e_i\in \Lambda, f^i\in \Lambda^*\}$ is a dual basis for $\Lambda_A$; see \cite[17.11]{BW03}.  

\begin{thm} \label{mainthm2}
Let $\Lambda$ be an algebra and let $A\subseteq \Lambda$ be a subalgebra such that $\Lambda_A$ is finitely generated projective. Denote by $\mathcal{C}=(A, \Lambda^* )$ the dual $A$-coring. Then the following assertions are equivalent:
\begin{enumerate}
\renewcommand{\theenumi}{\alph{enumi}}
\item\label{frobequ:1} The algebra extension $A\to \Lambda$ is a Frobenius extension. 
\item\label{frobequ:2} Induction $\Lambda\otimes_A -\colon \Mod{A}\to \LaMod$ is a Frobenius functor. 
\item\label{frobequ:3} Coinduction $\mathrm{Hom}_A({}_A\Lambda^*,-)\colon \Mod{A}\to \LaMod$ is a Frobenius functor.
\item\label{frobequ:4} The forgetful functor $\Forget\colon \LaMod\to\Mod{A}$ is a Frobenius functor.
\item\label{frobequ:5} $\mathcal{C}$ is a Frobenius coring.
\end{enumerate}
\end{thm}

\begin{proof}
The equivalences of \eqref{frobequ:1}, \eqref{frobequ:2} and \eqref{frobequ:4} are the standard characterisations of Frobenius extensions; see \cite{Mor65}. Since $\Lambda_A$ is finitely generated projective there is a natural isomorphism of induction and coinduction functors, which shows that \eqref{frobequ:3} is simply a rephrasing of \eqref{frobequ:2}. The same natural isomorphism induces the equivalence of categories of contramodules over the dual coring $\mathcal{C} = (\Lambda^*, A)$ and modules over the algebra $\Lambda$, which proves the equivalence of \eqref{frobequ:3} with \eqref{frobequ:4}. Finally, since $\Lambda_A$ is finitely generated projective the categories of modules over $\Lambda$ and comodules over its dual $A$-coring are mutually equivalent, the equivalence of \eqref{frobequ:1} and \eqref{frobequ:5} follows by Theorem~\ref{frobchar}.
\end{proof}

\subsection{Completing the proof of Theorem \ref{mainthm1}}

\subsubsection{\eqref{mainthm1:1} implies \eqref{mainthm1:2}:}
Assume that $\mathcal{C}=(A,V)$ is a Frobenius coring with  reduced Frobenius system $(\sigma\colon V\otimes_A V\to A,e\in V)$.  View $V$ as an $A^{\op}$-$\mathcal{C}^*$-bimodule by $avf= (f\otimes 1)\mu(va)$ and $V^*$ as an $A^{\op}$-$\mathcal{C}^*$-bimodule by $afg(v) = f*^rg(av)$. Then by \cite[27.13]{BW03} the map
\[
\delta_R\colon V\to V^*, \qquad w\mapsto \left(v\mapsto \sigma(w\otimes v)\right), 
\]
is an isomorphism of $A^{\op}$-$\mathcal{C}^*$-bimodules with inverse
\[
\delta_R^{-1}\colon V^* \to V, \qquad f\mapsto (f\otimes \id_V) (\mu (e)).
\]
Similarly, by \cite[27.10]{BW03} there is the isomorphism $\delta_L\colon V \to {}^*V$ of $A$-$(^* \mathcal{C})^{\op}$-bimodules. 
The required duality isomorphism  is obtained as the composition
$\delta = \delta_L\circ \delta_R^{-1}\colon V^*\to {}^*V$. Explicitly, writing $\mu(e) = \sum_i e'_i\otimes e''_i $,
\[
\delta(f)(v) = \sum_i \sigma(v\otimes f(e'_i)e''_i), \quad \delta^{-1}(g)(v) = \sum_i \sigma(e'_ig(e''_i)\otimes v),
\]
for all $f\in V^*, g\in {}^*V$ and $v\in V$. Using these explicit expressions for $\delta$ and \eqref{inclusions} for the inclusions $\incl$, we can compute for all $a\in A$, $v\in V$
\[
\delta(\incl(a))(v)= \sum_i \sigma(v\otimes \varepsilon(ae'_i)e''_i) = \sigma(v\otimes ae) = 
\sigma(va\otimes e)= \varepsilon(va) = \incl(a)(v).
\]
This proves the commutativity of the triangular diagram in statement \eqref{mainthm1:2} of Theorem~\ref{mainthm1}. It remains to prove that $\delta$ is a multiplicative map. To this end, we first set  
\[
\sum_i e'_i\otimes e''_i\otimes e'''_i = (\mu\otimes \id_V)\mu(e), \quad  \sum_i e'_i\otimes e''_i\otimes e'''_i\otimes e''''_i = (\mu\otimes \mu)\mu(e).
\]
Then, for all $f,g\in V^*$ and $v\in V$,
\[
\delta(f*^rg)(v) = \sum _i \sigma(v\otimes g(f(e'_i) e''_i)e'''_i).
\]
On the other hand 
\[
\begin{aligned}
    \delta(f)*^l\delta(g)(v) &= \sum_{i,j} \sigma\left(\sigma\left(v\otimes  g(e'_i) e''_i\right)e'''_i\otimes f(e'_j)e''_j\right)= \sum_{i,j} \sigma\left(v\otimes  g(e'_i) e''_i\sigma\left(e'''_i\otimes f(e'_j)e''_j\right)\right)\\
    &= \sum_{i,j} \sigma\left(v\otimes  g\left(e'_i\right)\sigma\left(e''_i\otimes f(e'_j)e''_j\right)e'''_j\right)\\
    &= \sum_{i,j} \sigma\left(v\otimes  g\left(e'_i\sigma\left(e''_i\otimes f(e'_j)e''_j\right)\right)e'''_j\right)\\
    &= \sum_{j} \sigma\left(v\otimes  g\left(\sigma\left(e\otimes f(e'_j)e''_j\right)e'''_j\right)e''''_j\right)\\
    &= \sum_{j} \sigma\left(v\otimes  g\left(\varepsilon\left( f(e'_j)e''_j\right)e'''_j\right)e''''_j\right)\\
    &= \sum_{j} \sigma\left(v\otimes  g\left( f(e'_j)e''_j\right)e'''_j\right) = \delta(f*^rg)(v).
\end{aligned}
\]

Here, the first,  third, fifth and sixth equalities all follow by \eqref{red.Frob}. The second one follows from the fact that $\sigma$ is a morphism of $A$-$A$-bimodules, the fourth one is a consequence of the right $A$-linearity of $g$ and the penultimate is obtained by the counit axiom.

\subsubsection{\eqref{mainthm1:1}, \eqref{mainthm1:7}, \eqref{mainthm1:8} and \eqref{mainthm1:9} are equivalent:}
Proposition \ref{frobfinite} shows that condition \eqref{mainthm1:1} implies all
conditions on the left or right $A$-module $V$ in the other statements.

The equivalence of \eqref{frobchar:1} and \eqref{frobchar:2} in Theorem \ref{frobchar} is the
equivalence of \eqref{mainthm1:1} and \eqref{mainthm1:7}. The equivalence of \eqref{frobchar:1}, \eqref{frobchar:4} and \eqref{frobchar:5} in Theorem \ref{frobchar} is the
equivalence of \eqref{mainthm1:1}, \eqref{mainthm1:8} and \eqref{mainthm1:9}. 

\subsubsection{\eqref{mainthm1:5}, \eqref{mainthm1:6} and \eqref{mainthm1:7} are equivalent:}
Induction in \eqref{mainthm1:5} and coinduction in \eqref{mainthm1:6} are left and right adjoint to the
forgetful functor in \eqref{mainthm1:7}. Hence, one of the functors is Frobenius if and
only if the other two functors are so.

\subsubsection{\eqref{mainthm1:2}, \eqref{mainthm1:3} and \eqref{mainthm1:4} are equivalent:}
\eqref{mainthm1:2} and \eqref{mainthm1:3} are equivalent by the identification of Burt--Butler algebras with opposite dual algebras.
Statement \eqref{mainthm1:3} is contained in \eqref{mainthm1:4} and conversely an isomorphism of algebras induces a
Morita equivalence. Commutativity of the diagrams then follows from the
isomorphism $\delta$ fixing $A$. All this is a special case of lifting functors and natural transformations to functors between Eilenberg--Moore categories of monads; see \cite[Chapter~2]{Boh18} for a concise, yet very accessible, exposition. 

\subsubsection{\eqref{mainthm1:4} implies \eqref{mainthm1:8}:}
Suppose natural self-duality \eqref{mainthm1:4} is satisfied. 

Induction $R \otimes_A -$ from $\Mod{A}$ to $\Mod{R}$
always is  the left adjoint functor to restriction 
$\Mod{R} \rightarrow \Mod{A}$ and
restriction from $\Mod{L}$ to $\Mod{A}$ always
is  the left adjoint functor to coinduction $\Hom_A({}_AL,-)$.

By \eqref{mainthm1:4}, $\Hom_A(L,-)$ equals $F \circ (R \otimes_A-)$, which is the left adjoint functor
to restriction $\Mod{R} \rightarrow \Mod{A}$ precomposed with $F^{-1}$.
That composition equals restriction from $\Mod{L}$ to $\Mod{A}$ follows by \eqref{mainthm1:4}.
Hence, left and right adjoints of $\Hom_A({}_AL,-)$ are isomorphic, which means
that $\Hom_A({}_AL,-)$ is a Frobenius functor. This implies that restrictions
from $\Mod{R}$ or $\Mod{L}$ to $\Mod{A}$ are Frobenius functors, too, and
thus induction $R \otimes_A -$ is a Frobenius functor as well, which is equivalent to say that $A\subseteq R$ is a Frobenius extension. Since ${}_AV$ is assumed to be finitely generated projective, the statement \eqref{mainthm1:8} follows.

\subsubsection{\eqref{mainthm1:8} is equivalent to \eqref{mainthm1:10} and \eqref{mainthm1:9} is equivalent to \eqref{mainthm1:11}:}
The first equivalence is contained in \cite[27.10]{BW03} while the second one is contained in \cite[27.13]{BW03}.

\bigskip

This completes the proof of Theorem~\ref{mainthm1}.

\subsection{Trivial extensions and coextensions}\label{sec:nec_cond}

In this subsection, we show that the additional assumption on $V$ being finitely generated projective as a left $A$-module at several places in Theorem \ref{mainthm1} is necessary by providing counterexamples in the setting of trivial extensions and coextensions.

\subsubsection{~} Given an algebra $A$ and an $A$-$A$-bimodule $M$, one can construct in a well-known way a new algebra $\Triv(A,M)$, the trivial extension of $A$ and $M$, as follows:
\[\Triv(A,M)=A\oplus M\]
with multiplication $(a,m)\cdot (b,n)=(ab, an+mb)$. This construction arises naturally in many places in representation theory. In the special case $M=D(A)$, the dual vector space of $A$, it shows that every finite-dimensional algebra over a field can be realised as a quotient of a symmetric algebra, and secondly, for $A$ of finite global dimension, Happel proved that the derived category of $A$ is equivalent to the category of graded modules over $\Triv(A,D(A))$. Here we use the trivial extension to indicate that \eqref{mainthm1:1} $\Leftrightarrow$ \eqref{mainthm1:5} in Theorem \ref{mainthm1} fails when dropping the assumption that $V$ is finitely generated projective as a left $A$-module. Note that without the condition that $\Triv(A,M)$ is finitely generated projective over $A$ on one side (equivalently that $M$ is finitely generated projective on that side), one cannot in general even take the corresponding dual without further argument. 
For the counterexample, let $A=\Bbbk A_2$ be the path algebra of the $A_2$-quiver over a field $\Bbbk$, that is the algebra of upper triangular $2\times 2$-matrices. Let $M$ be the non-projective, but injective simple module. Then $\Hom_A(M,A)=0$, so that $\Hom_A(\Triv(A,M),A)=\Hom_A(A\oplus M,A)\cong A$ has an obvious $A$-coring structure, namely the principal one where the comultiplication  $A\to A\otimes_A A$ is the canonical isomorphism and the counit $A\to A$ is the identity. This is a Frobenius $A$-coring (with $\eta$ being the identity and $\pi$ being the inverse of the comultiplication). However, $\Triv(A,M)$ cannot be a Frobenius extension of $A$ as it is not finitely generated projective.  

\subsubsection{~} Dually to the trivial extension algebra, Casta\~no Iglesias, D\u asc\u alescu, and N\u ast\u asescu defined the trivial coextension coalgebra in \cite[Section 4]{CDN04}. Given a coalgebra $D$ and a $D$-$D$-bicomodule $N$, it is defined as 
\[\Triv(D,N)=D\oplus N\]
with comultiplication given by 
\begin{equation}
(d,n)\longmapsto \mu(d)+\rho(n)+\lambda(n),\label{trivialcoextension_formula}
\end{equation}
where $\mu$ is the comultiplication on $D$ and $\lambda$ and $\rho$ denote the left and right coactions of $D$ on $N$. This construction generalises to $A$-corings instead of coalgebras. Namely, if $\mathcal{C}=(A,V)$ is an $A$-coring and $N$ is a $\mathcal{C}$-$\mathcal{C}$-bicomodule, then \eqref{trivialcoextension_formula} defines an $A$-coring $\Triv(\mathcal{C},N)=V\oplus N$. Applying the construction to the principal $A$-coring $A$, we obtain an $A$-coring structure on $A\oplus N$. In this case, a $\mathcal{C}$-$\mathcal{C}$-bicomodule is just the same as an $A$-$A$-bimodule. In contrast to the preceding example of the trivial extension, here we do not have to worry about the dual being defined, the left (or right) algebra of an $A$-coring is an algebra with $A$ as a subalgebra. 
Similarly to the case above take $A=\Bbbk A_2$ and $\mathcal{C}=A$ to be the trivial coring (with comultiplication the inverse of standard multiplication and counit being the identity). Then $\mathcal{C}$-comodules can be identified with $A$-modules and taking the comodule $N$ corresponding to the simple non-projective $A$-module, we obtain that the (right) algebra $\Hom_A(\Triv(\mathcal{C},N),A)$ is just the algebra $A$, which is a Frobenius extension of itself. However, the coring $\Triv(\mathcal{C},N)$ cannot be a Frobenius coring since $N$ is not finitely generated projective on the left. Thus, the condition of finitely generated projective in Theorem \ref{mainthm1} \eqref{mainthm1:8} is necessary.

\subsubsection{~} On the other hand, taking $A=\Bbbk[x]/(x^2)$, $\mathcal{C}=A$ the trivial coring as above and $N$ to be the comodule corresponding to the trivial $A$-module $\Bbbk$, we obtain that the (right) algebra $\Hom_A(\Triv(\mathcal{C},N),A)$ is isomorphic to $\Triv(\mathcal{C},N)$ as an $A$-$R$-bimodule. However, again $\Triv(\mathcal{C},N)$ is not finitely generated projective as a (left) $A$-module. This shows that the condition of finitely generated projective in Theorem \ref{mainthm1} \eqref{mainthm1:10} is necessary. 

\section{Proof of Theorem~\ref{mainthm3}}\label{sec:proofB}
\label{sec:comparing-dualities}

Throughout this section, the underlying commutative ring is assumed to be a field. In parallel work, but unrelated to that of Burt and Butler on bocses, Ringel discovered, in the context of quasi-hereditary
algebras $R$, a tilting-cotilting module $T$ whose opposite endomorphism algebra $R'=\End_R(_R T)^{\op}$ is
quasi-hereditary, too. This algebra $R'$  is now called Ringel dual of $R$
and the connection between $R$ and $R'$, relating representations
through functors involving $T$, is
called Ringel duality. This, and in particular the case of Ringel
self-duality, has found many applications in representation theory, in
algebraic Lie theory and elsewhere, and Ringel duality also has been
extended to the larger class of standardly stratified algebras, cf. \cite{AHLU00}. The problem
to decide which quasi-hereditary or standardly stratified algebras are
Ringel self-dual remains, however, wide open.

In \cite{KKO14} it has been shown that every quasi-hereditary algebra
is Morita equivalent, as quasi-hereditary algebra, to a quasi-hereditary
algebra $R$ with a subalgebra $A$, called regular exact Borel subalgebra,
such that the pair $(A,R)$ corresponds to a coring $(A,V)$. In this case, the $\Bbbk$-dual of $V$ is, up to direct summands, the characteristic tilting module $T$ (see \cite[Proof of Theorem 11.5]{KKO14}).
Moreover, Ringel duality \cite[Section 5]{Rin91} has been related to Burt--Butler duality. In \cite{BKK20}, the coring perspective
has been further used to understand quasi-hereditary algebras and their
bocses. 

A motivation for this article has been to continue \cite{KKO14}
and \cite{BKK20} and use corings in order to compare Ringel self-duality with
coring (or Burt--Butler) self-duality, which uses the same pair of
quasi-hereditary algebras, but not the same functors as Ringel duality.
Therefore, we continue to work with the algebra extension $A \subseteq R$ that has
also been used in \cite{KKO14,BKK20} and ask when it is a Frobenius extension, that is, when it is self-dual as a
coring.

\subsection{Quasi-hereditary algebras and standardly stratified algebras.}
We start by recalling the definition of a quasi-hereditary algebra and the more general notion of a standardly stratified algebra. 

\begin{defn}
Let $R$ be a finite-dimensional algebra. Let $Q_0$ be an indexing set of the simple $R$-modules, which are denoted by $L_\mathtt{i}$. Their projective covers are denoted by $P_{\mathtt{i}}$. Choose a partial order $\leq$ on $Q_0$. 
\begin{itemize}
\item Then, the \emphbf{standard module} $\Delta_\mathtt{i}$ is the largest quotient of the indecomposable projective $P_\mathtt{i}$ such that all of its composition factors are of the form $L_\mathtt{j}$ with $\mathtt{j}\leq \mathtt{i}$.
\item The \emphbf{proper standard module} $\overline{\Delta}_\mathtt{i}$ is the largest quotient of the standard module $\Delta_\mathtt{i}$ such that $[\overline{\Delta}_\mathtt{i}:L_\mathtt{i}]=1$.
\item The algebra $R$ is called \emphbf{left standardly stratified} with respect to $\leq$ if, for all $\mathtt{i} \in Q_0$, the kernel of $P_\mathtt{i}\twoheadrightarrow \Delta_\mathtt{i}$ has a filtration whose subquotients are isomorphic to $\Delta_\mathtt{j}$ with $\mathtt{j}> \mathtt{i}$.
\item The algebra $R$ is called \emphbf{right standardly stratified} with respect to $\leq$ if, for all $\mathtt{i} \in Q_0$, the kernel of $P_\mathtt{i}\twoheadrightarrow \overline{\Delta}_\mathtt{i}$ has a filtration whose subquotients are isomorphic to $\overline{\Delta}_\mathtt{j}$ with $\mathtt{j}\geq \mathtt{i}$. 
\item The algebra $R$ is called \emphbf{quasi-hereditary} if it is left standardly stratified and in addition $\Delta_\mathtt{i}=\overline{\Delta}_\mathtt{i}$ for all $\mathtt{i}\in Q_0$. 
\end{itemize}
\end{defn}

Examples of quasi-hereditary algebras are blocks of BGG category $\mathcal{O}$, Schur algebras of symmetric groups, and algebras of global dimension less than or equal to two. For a general introduction to the theory of quasi-hereditary algebras, see \cite{CPS88, DR92, KK99}. 

A key player in the theory of quasi-hereditary algebras is the subcategory $\mathcal{F}(\Delta)$ of modules admitting a filtration by the standard modules, defined by 
\[\mathcal{F}(\Delta)=\{M\in \modu{R}\mid \exists\,\, 0=M_0\subset M_1\subset \dots\subset M_t=M\text{ with }M_j/M_{j-1}\cong \Delta_{\mathtt{i}_j}\}.\]

Dually, one defines costandard modules $\nabla_\mathtt{i}$ as highest weight submodules of the indecomposable injective modules and their proper submodules $\overline{\nabla}_\mathtt{i}$. Analogously to $\mathcal{F}(\Delta)$ one defines the  subcategories $\mathcal{F}(\overline{\Delta})$, $\mathcal{F}(\nabla)$, and $\mathcal{F}(\overline{\nabla})$. For a left (resp.~right) standardly stratified algebra, the category $\mathcal{F}(\Delta)$ (resp.~$\mathcal{F}(\overline{\Delta})$) is the prototypical example of a \emphbf{resolving subcategory} of the module category, that is, it is closed under extensions and kernels of epimorphisms, direct summands, and furthermore contains the projective modules, see \cite[Theorem 1.6]{AHLU00}. The category $\mathcal{F}(\overline{\nabla})$ (resp.~$\mathcal{F}(\nabla)$) satisfies dual properties: it is coresolving in case $R$ is left (resp.~right) standardly stratified. It turns out that for a left standardly stratified algebra, the intersection of these subcategories $\mathcal{F}(\Delta)\cap \mathcal{F}(\overline{\nabla})$ is precisely the additive closure of a tilting module, the \emphbf{characteristic tilting module} $T$, see \cite[Proposition 2.2]{AHLU00} and \cite[Theorem 5]{Rin91}. Its endomorphism algebra $R'$ is right standardly stratified with respect to a naturally induced partial order, and is called the \emphbf{Ringel dual} of $R$. As the term `duality' suggests, an analogous construction exists for right standardly stratified algebras. The intersection $\mathcal{F}(\overline{\Delta})\cap \mathcal{F}(\nabla)$ in $\modu{R'}$ equals the additive closure of a cotilting module, the \emphbf{characteristic cotilting module} $T'$. Its endomorphism algebra is naturally left standardly stratified and Morita equivalent to $R$, with the equivalence preserving the left standardly stratified structure. In many natural examples of quasi-hereditary algebras, such as blocks of BGG category $\mathcal{O}$ (see \cite[Struktursatz 9]{Soe90}) and certain Schur algebras (see \cite[Proposition 3.7]{Don93} and \cite[Theorem 27]{EH02}), one has observed the phenomenon of \emphbf{Ringel self-duality}, i.e.\ that the Ringel dual of $R$ is actually Morita equivalent to $R$ itself, with the equivalence transporting the quasi-hereditary structure.

\subsection{Regular exact subalgebras.}

Motivated by the PBW theorem in Lie theory and by the example of the BGG category $\mathcal{O}$, \cite{Koe95} introduced a class of subalgebras of quasi-hereditary algebras, called exact Borel subalgebras, which turn out to control the subcategory $\mathcal{F}(\Delta)$. Existence of such exact Borel subalgebras was proved in \cite{KKO14}, and existence of appropriate generalisations for left (or right) standardly stratified algebras was obtained in \cite{BPS23} and \cite{Got24}. The work of \cite{KKO14} connects the theory of exact Borel subalgebras to the theory of corings. In this context, Kleiner and Roiter introduced an additional assumption of regularity, see \cite{KR77}. This condition was later dualised to the context of exact Borel subalgebras, see \cite{BKK20, KM25}. It turns out that the notion of a regular exact subalgebra extends beyond quasi-hereditary algebras and left or right standardly stratified algebras, see \cite{Kul25}. Rodriguez Rasmussen proved in \cite{RR25} (cf. \cite{KM25}) that if $A$ and $B$ are two basic regular exact subalgebras (the notion we define next) of a finite-dimensional algebra $R$, then there exists an invertible element $r\in R$ such that $B=r^{-1}Ar$, i.e.\ basic regular exact subalgebras are unique up to conjugation, which can be seen as a generalisation of an analogue of the theorem that Borel subalgebras of Lie algebras are unique up to an inner automorphism. It can also be seen as an analogue of the Wedderburn--Malcev theorem, for basic regular exact subalgebras. 

\begin{defn}\label{definition-regular-exact}
Let $R$ be a finite-dimensional algebra. A subalgebra $A\subseteq R$ is called a \emphbf{regular exact subalgebra} if
\begin{itemize}
\item $R$ is projective when viewed as a right $A$-module, and 
\item The induction functor $R\otimes_A -$ induces isomorphisms of vector spaces
\begin{equation}
\label{eq:regularity}\Ext^k_A(L_\mathtt{i}, L_\mathtt{j})\to \Ext^k_R(R\otimes_A L_\mathtt{i}, R\otimes_A L_\mathtt{j})
\end{equation}
for all $\mathtt{i}$ and $\mathtt{j}$, and all $k\geq 1$. 
\end{itemize}
If $R$ is left standardly stratified, we call a regular exact subalgebra $A\subseteq R$ a \emphbf{regular exact Borel subalgebra} if in addition the following two conditions are satisfied:
\begin{itemize}
\item $R$ and $A$ have the same indexing set of simple modules and $R\otimes_A L_\mathtt{i}^A\cong \Delta_\mathtt{i}^R$, and
\item $\End_A(L_\mathtt{i}^A)\cong \End_R(L_\mathtt{i}^R)$ as $\Bbbk$-algebras. 
\end{itemize}
If $R$ is right standardly stratified, we call a regular exact subalgebra $A\subseteq R$ a \emphbf{regular exact Borel subalgebra} if in addition the following two conditions are satisfied:
\begin{itemize}
\item $R$ and $A$ have the same indexing set of simple modules and $R\otimes_A L_\mathtt{i}^A\cong \overline{\Delta}_\mathtt{i}^R$, and 
\item $\End_A(L_\mathtt{i}^A)\cong \End_R(L_\mathtt{i}^R)$ as $\Bbbk$-algebras.
\end{itemize}
\end{defn}

\begin{rmk}
\begin{itemize}
\item Note that the last condition is not present in the original article \cite{Koe95} (nor in subsequent papers \cite{KKO14, BKK20, BPS23, Got24}) since it is assumed there that all algebras are over an algebraically closed field. It appears for the first time in \cite{Con21}. 
\item In this text we have chosen to focus on regular exact Borel subalgebras. For this reason, the directedness or one-cycle directeness conditions that usually appear when defining an exact Borel subalgebra (\cite{Koe95,BPS23,Got24}) are omitted from our definition of a regular exact Borel subalgebra, as they follow from the isomorphism in \eqref{eq:regularity}. Our regular exact Borel subalgebras of a left (resp.~right) standardly stratified algebra are then left (resp.~right) standardly stratified with simple standard modules. In particular, exact Borel subalgebras of left standardly stratified algebras are always quasi-hereditary. What we would call an exact Borel subalgebra of a right standardly stratified algebra is called a proper Borel subalgebra in \cite{Got24}. 
\end{itemize}
\end{rmk}

\subsection{Examples}
The context for Theorem \ref{mainthm3} has now been established. Before proceeding to the proof techniques, we present two examples that illustrate this setting.
\begin{ex}
To give a flavour of the results we are going to obtain in this section, in this example we study the properties of a particular kind of Frobenius extension, namely one arising from tensoring with a Frobenius algebra and show explicitly that it cannot be regular in the above sense. Let $A$ be an arbitrary finite-dimensional algebra and let $B$ be a (finite-dimensional) Frobenius algebra, i.e.\ assume that the canonical map $\Bbbk \rightarrow B$ is a Frobenius extension. Setting $\Lambda=B\otimes_\Bbbk A$, the extension $A\subseteq \Lambda$ is a Frobenius extension as observed in \cite[Beispiel 3 b)]{Kas54}. An application of the K\"unneth formula (see e.g.\ \cite[Lemma 4.10]{Pas19}) yields that 
\[\Ext_{\Lambda}^k(\Lambda\otimes_A M,\Lambda\otimes_A N)\cong B\otimes_\Bbbk \Ext^k_A(M,N)\]
for all $A$-modules $M$ and $N$, where the map on $\Ext$-groups induced by the (exact) induction functor $\Lambda\otimes_A -$ can be identified with 
\begin{align*}
\Ext^k_A(M,N)&\to B\otimes_\Bbbk \Ext^k_A(M,N)\\
\eta&\mapsto 1\otimes \eta.
\end{align*}
 This map is always injective, but surjective only if $B\cong \Bbbk$. Thus, the algebra extension $A\subseteq \Lambda$ is regular only in the trivial case $B\cong \Bbbk$. 
 \end{ex}

\begin{ex}
Let $R$ be the algebra $\begin{pmatrix}\Bbbk[x]/(x^2)&\Bbbk[x]/(x^2)\\0&\Bbbk[x]/(x^2)\end{pmatrix}$ with partial order so that $\Delta_\mathtt{1}=\begin{pmatrix}\Bbbk[x]/(x^2)\\0\end{pmatrix}$ and $\Delta_\mathtt{2}=\begin{pmatrix}0\\\Bbbk[x]/(x^2)\end{pmatrix}$. In this case, the proper standard modules are the simple $R$-modules. Direct checking yields that $R$ is both left and right standardly stratified. As a right standardly stratified algebra, $R$ is its own regular exact Borel subalgebra. On the other hand, as a left standardly stratified algebra, a regular exact Borel subalgebra $A$ would be given by the path algebra of the Kronecker quiver, as $\dim_\Bbbk \Ext^1_R(\Delta_\mathtt{1},\Delta_\mathtt{2})=2$. However, this cannot be realised as a subalgebra of $R$. The path algebra of the Kronecker quiver is an exact Borel subalgebra of the Morita equivalent algebra $\End_R(R\oplus P_\mathtt{2})^{\op}$. In fact, suppose that $A$ is an exact Borel subalgebra of left standardly stratified algebra $R'$ Morita equivalent to $R$ as a left standardly stratified algebra. The projective $A$-module $P_\mathtt{2}^A$ is simple and therefore induces to $\Delta_\mathtt{2}$. The projective $A$-module $P_\mathtt{1}^A$ has dimension vector $(1,2)$. It therefore induces to a projective $R'$-module filtered by one $\Delta_\mathtt{1}$ and two $\Delta_\mathtt{2}$. The indecomposable projective $P_\mathtt{1}$ has a $\Delta$-filtration with one $\Delta_\mathtt{1}$ and and one $\Delta_\mathtt{2}$. Therefore $R'\otimes_A P_\mathtt{1}^A\cong P_\mathtt{1} \oplus P_\mathtt{2}$. In particular, this shows that unlike for quasi-hereditary algebras, for left standardly stratified algebras with at most two simple modules, not every algebra has a basic regular exact Borel subalgebra, cf. \cite[Example 4.20]{Con21b} and \cite[Theorem 4.58]{Kul17b}. 
\end{ex}

\subsection{Relative homological algebra}
As realised in \cite{KKO14,BKK20} (under the assumption of an algebraically closed field $\Bbbk$), for quasi-hereditary algebras also the notion of Ringel duality can be understood via regular exact Borel subalgebras. Given a quasi-hereditary algebra $R$ with a regular exact Borel subalgebra $A\subseteq R$, let $\mathcal{C}= (A, V={\Hom_{A}(R_A,A_A)})$ be its dual coring. Then the Ringel dual of $R$ is given by the left Burt--Butler algebra $L$ of $\mathcal{C}$. In addition, Ringel duality can also be understood as a special case of dg Koszul duality for the dg algebra $T_A(V)$, see \cite{Ovs06,BK18}. In view of this, it seems natural to ask whether the phenomenon of Ringel self-duality observed in examples can also be understood as an instance of Burt--Butler self-duality as described earlier. Unfortunately, this is almost never the case. In the remainder of this section, we will explain why and provide a homological obstruction in the more general case of regular exact subalgebras. 

The homological obstruction is in terms of Hochschild's relative homological algebra \cite{Hoc56}, more precisely it is given in terms of the finiteness of the relative global dimension. Relative homological algebra continues to be a fruitful tool in representation theory and beyond, see e.g.~\cite{XX13, IMP24}. We recall some notions needed in this context.  

\begin{defn}
Let $A\subseteq R$ be an algebra extension. An $R$-module $M$ is called \emphbf{relative projective (with respect to $A\subseteq R$)} or \emphbf{$(R,A)$-projective} if every epimorphism of $R$-modules $N\twoheadrightarrow M$, which splits when restricted to $A$-modules, also splits as $R$-modules.   
\end{defn}

There are two extreme cases of the definition. When $A=R$, then every $R$-module is $(R,A)$-projective. On the other hand if $A$ is semisimple, then the $(R,A)$-projective modules coincide with the projective $R$-modules since every epimorphism of $A$-modules splits. 

\begin{defn}
    Let $A\subseteq R$ be an algebra extension. An exact sequence of $R$-modules
   \[\begin{tikzcd}
	\cdots & {M_i} & {M_{i-1}} & \cdots
	\arrow[from=1-1, to=1-2]
	\arrow["{t_i}", from=1-2, to=1-3]
	\arrow[from=1-3, to=1-4]
\end{tikzcd}\]
is \emphbf{$(R,A)$-exact} if, for every $i$, the kernel of $t_i$ is a direct $A$-module summand of $M_i$.
\end{defn}

Once relative projective modules are defined, one can develop relative homological algebra as was started by Hochschild in \cite{Hoc56}. For instance, the relative $\Ext$-functor $\Ext^k_{(R,A)}(M,N)$ is given by taking a relative projective resolution of $M$, applying $\Hom_R(-,N)$ and taking homology. The classical descriptions of relative projectivity carry over to the relative setting. 

\begin{lem}
Let $A\subseteq R$ be an algebra extension. Then the following are equivalent for an $R$-module $M$:
\begin{enumerate}[(1)]
\item $M$ is relative projective;
\item $M$ is a direct summand of a module of the form $R\otimes_A N$ for some $A$-module $N$;
\item $\Ext_{(R,A)}^1(M,-)=0$;
\item The map $\Ext^1_R(M,X)\to \Ext^1_A(M,X)$ is injective for any $R$-module $X$. 
\end{enumerate} 
\end{lem}

\begin{proof}
The equivalence of (1) and (2) is e.g.\ stated in \cite[Lemma 1.1]{The85}. For the equivalence of (1) and (4) see \cite[Remark 2.4]{Zim20}. The equivalence between (1) and (3) follows e.g.\ from \cite[Proposition 3.1, Theorem 3.9]{Hol08}.
\end{proof}

 In addition, Hochschild introduced an analogue of global dimension in this context, namely relative global dimension, see \cite[p.~258]{Hoc56}.

\begin{defn}
Let $R$ be a finite-dimensional algebra and let $A\subseteq R$ be a subalgebra. 
\begin{itemize}
\item The \emphbf{relative projective dimension} of an $R$-module $M$ (with respect to the subalgebra $A$) is the infimum over all $n\in \mathbb{N} \cup \{\infty\}$ such that there exists an exact sequence of $R$-modules
\[
\begin{tikzcd}
0\arrow{r} &P_n\arrow{r} &\dots\arrow{r} &P_0\arrow{r} &M\arrow{r} &0
\end{tikzcd}
\]
which is $(R,A)$-exact, i.e.\ the kernel of each map $P_j\to P_{j-1}$ is a direct summand of $P_j$ as an $A$-module, and where $P_i$ is an $(R,A)$-projective module. 
\item The \emphbf{relative global dimension} $\gldim(R,A)$ of $R$ with respect to $A$ is defined to be the supremum of the relative projective dimensions of all $R$-modules. 
\end{itemize}
\end{defn}

As in the case of relative projectivity, relative global dimension can be characterised by vanishing of the relative $\Ext$-functor. 

\begin{lem}
Let $R$ be a finite-dimensional algebra and let $A\subseteq R$ be a subalgebra. Then the following agree:
\begin{enumerate}[(1)]
\item The relative global dimension of $R$ with respect to $A$,
\item the supremum of all $n$ such that there exist $R$-modules $M$ and $N$ with 
\[\Ext^n_{(R,A)}(M,N)\neq 0.\]
\end{enumerate}
\end{lem}

\begin{proof}
This follows from \cite[Proposition 3.1, Theorem 3.9]{Hol08}.
\end{proof}

We now return to the key player of this article, namely Frobenius extensions, to determine its intersection with regular exact Borel subalgebras. It is well-known that Frobenius algebras are self-injective and are therefore either semisimple or of infinite global dimension. In \cite[Theorem 10]{Hir59}, Hirata has proved the following analogue for Frobenius extensions.

\begin{thm}\label{hirata_theorem}
Let $A\subseteq R$ be a Frobenius extension. Then $\gldim(R,A)$ is either $0$ or $\infty$. 
\end{thm}

Hirata's theorem is the key result needed
to deduce that regular exact subalgebras of
right standardly stratified algebras (or of algebras satisfying similar technical properties) are typically not Frobenius extensions.

\subsection{Completing the proof of Theorem \ref{mainthm3}}

For proving Theorem \ref{mainthm3}, we use the theory of $A_\infty$-algebras. The notion of an $A_\infty$-algebra is a generalisation of the notion of a differential graded algebra that allows for transport of structure under quasi-isomorphisms. In particular, the cohomology of a differential graded algebra carries the structure of an $A_\infty$-algebra, a statement known as Kadeishvili's theorem. This applies in particular to the $\Ext$-algebra $\Ext^{\geq 0}_A(M,M)$ of a module $M$ over an algebra $A$. To the $A_\infty$-algebra $\Ext^{\geq 0}_A(M,M)$ one can associate an $A_\infty$-category $\twmod \Ext^{\geq 0}_A(M,M)$, the category of twisted modules or twisted stalks, that corresponds to taking the closure of $M$ under extensions in the category of $A$-modules. More precisely, its homotopy category $H^0(\twmod \Ext^{\geq 0}_A(M,M))$ is equivalent to the extension closure of $M$, see e.g.\ \cite[Section 7]{Kel01}.  For an introduction to the theory of $A_\infty$-algebras, we refer the reader to \cite{Kel01, Kel02, LPWZ04, KM25}.

Recall the statements \eqref{mainthm3:a} and \eqref{mainthm3:b} in Theorem \ref{mainthm3}.

{\it \begin{enumerate}[(a)]
\item Let $A\subseteq R$ be an extension of finite-dimensional algebras and
  suppose $A$ is a regular exact subalgebra such that the category  $\Ind_A^R$ of induced modules is a
  resolving subcategory and the relative global dimension of $R$ with respect to $A$ is finite. Suppose further
  that $A\subseteq R$ is a Frobenius extension. 
  Then $R$ is relative semisimple in the sense that every $R$-module $M$ is
  relative projective. In fact, $M$ is isomorphic to a module of the form
  $R\otimes_A N$ for some $A$-module $N$. 
\item Keep the assumptions of (a) and assume in addition that $R\otimes_A -$ 
  reflects isomorphisms between simple $A$-modules and that 
  $\End_A(L_\mathtt{i})\cong\End_R(R\otimes_A L_\mathtt{i})$ for all $i$. Then
  $A=R$. 
\end{enumerate}
}
\begin{proof}
  (a) Since the relative global dimension of $A\subseteq R$ is finite and $A\subseteq R$ is a Frobenius extension, Theorem~\ref{hirata_theorem}  implies that the relative global dimension of $A\subseteq R$ is zero. Since by assumption the category $\Ind_A^R$ is closed under direct summands, this means that every $R$-module is an induced $R$-module. In particular, this holds for the simple $R$-modules. Furthermore, since $R\otimes_A -$ is faithful (cf. \cite[Proof of Theorem 4.16 using Lemma 3.1]{Con21b}), a simple module can only be induced from a simple module since otherwise it would have a non-zero submodule.

(b) Since in addition $R\otimes_A -$ reflects isomorphisms between simple $A$-modules, it follows that $R\otimes_A -$ induces a bijection between (isomorphism classes of) simple $A$-modules and (isomorphism classes of) simple $R$-modules. Since $A\subseteq R$ is a regular exact subalgebra, we obtain an isomorphism of $A_\infty$-algebras $\Ext^{>0}_A(\mathbb{L},\mathbb{L})\to \Ext^{>0}_R(R\otimes_A \mathbb{L},R\otimes_A \mathbb{L})$, where $\mathbb{L}=\bigoplus_{\mathtt{i}}L_\mathtt{i}$; see \cite[Theorem~3.21]{KM25} (note that the proof therein doesn't need the assumption of $A$ being basic). By assumption, this isomorphism extends to degree $0$, i.e.\ there is an isomorphism of $A_\infty$-algebras $\Ext^{\geq 0}_A(\mathbb{L},\mathbb{L})\to \Ext^{\geq 0}_R(R\otimes_A \mathbb{L},R\otimes_A \mathbb{L})$. This isomorphism induces an equivalence 
\[
\begin{tikzcd}
H^0(\twmod \Ext^{\geq 0}_A(\mathbb{L},\mathbb{L}))\arrow{r} &H^0(\twmod \Ext^{\geq 0}_R(R\otimes_A \mathbb{L}, R\otimes_A \mathbb{L})).
\end{tikzcd}
\] According to \cite[Corollary 5.17, cf. (3) on p.4]{RR25}, there thus is a commutative diagram
\[
\begin{tikzcd}
\modu{A}\arrow{r}{R\otimes_A -} &\modu{R}\\
H^0(\twmod\Ext^{\geq 0}_A(\mathbb{L},\mathbb{L}))\arrow{r}\arrow{u}&H^0(\twmod\Ext^{\geq 0}_R(R\otimes_A \mathbb{L}, R\otimes_A \mathbb{L}))\arrow{u}
\end{tikzcd}
\]
where all arrows are equivalences. In particular, $\End_A(A)\to \End_R(R\otimes_A A)$ is an isomorphism, whence $A=R$. 
\end{proof}

In case that $A\subseteq R$ is a regular exact Borel subalgebra of a quasi-hereditary algebra, one can simplify the above proof noting that even the isomorphism of graded vector spaces $\Ext^{>0}_A(\mathbb{L},\mathbb{L})\to \Ext^{>0}_R(R\otimes_A \mathbb{L}, R\otimes_A \mathbb{L})$ implies that $R$ is directed and conclude that $A=R$ using \cite[Lemma 3.7]{RR25}.

\begin{ex}
The last assumption $\End_A(L_\mathtt{i})\cong \End_R(R\otimes_A L_\mathtt{i})$ for all $\mathtt{i}$ is necessary to conclude $A=R$ as the example of $A=\mathbb{R}\subseteq R=\mathbb{C}$ shows. 
\end{ex}

In the remainder of this section, we aim to prove \eqref{mainthm3:c} in Theorem \ref{mainthm3}, which we remind the reader is as follows:

\medskip

\textit{
(c) Let $R$ be a right standardly stratified algebra and let $A\subseteq R$ be a regular exact Borel subalgebra. Suppose further that $A\subseteq R$ is a Frobenius extension. Then $A=R$ and this algebra is semisimple.}

\medskip
We note that the statement has two conclusions, the former claiming that $A=R$ and the latter that $R$ is semisimple. The latter has the former as a consequence by the condition that $\End_A(L_\mathtt{i})\cong \End_R(R\otimes_A L_{\mathtt{i}})$. Nevertheless we start by providing a proof of the former to illustrate how to use part \eqref{mainthm3:b} in Theorem \ref{mainthm3} in practice. 

The key ingredient is the finiteness of relative global dimension in \eqref{mainthm3:a}. To prove this, we note that the relative global dimension coincides with the resolution dimension of the category of modules with respect to the category of modules filtered by proper standard modules. The resolution dimension is an extension of relative global dimension (and classical global dimension) to more general subcategories, including the category $\mathcal{F}(\overline{\Delta})$, see e.g.\ \cite{Zhu13}. Such subcategories are ubiquitous in representation theory.

\begin{defn}
Let $\mathcal{X}\subseteq \modu{A}$ be a full subcategory. For a module $M\in \modu{A}$ and $X\in \mathcal{X}$, a morphism $f\colon X\to M$ is called an \emphbf{$\mathcal{X}$-precover} or a \emphbf{right $\mathcal{X}$-approximation} if for every $X'\in \mathcal{X}$ and any morphism $g\colon X'\to M$ there exists a morphism $h\colon X'\to X$ such that $g=fh$. A subcategory $\mathcal{X}\subseteq \modu{A}$ is called \emphbf{contravariantly finite} if for every $M\in \modu{A}$ there exists a right $\mathcal{X}$-approximation. For a subcategory $\mathcal{X}\subseteq \modu{A}$, the minimal $n$ such that there exists an exact sequence  
\[
\begin{tikzcd}
0\arrow{r} &X_n\arrow{r} &\dots\arrow{r} &X_0\arrow{r} &M\arrow{r} &0    
\end{tikzcd}\]
with $X_i\in \mathcal{X}$ is called the \emphbf{$\mathcal{X}$-resolution dimension} of $M$. The supremum of the $\mathcal{X}$-resolution dimensions over all $M\in \modu{A}$ is called the \emphbf{$\mathcal{X}$-resolution dimension} of $\modu{A}$.
\end{defn}

In this context, the homological criterion of finiteness of global resolution dimension appears in the correspondence between cotilting modules and contravariantly finite resolving subcategories of module categories, established by Auslander and Reiten, which we recall next.

\begin{thm}[{\cite[Theorem 5.5]{AR91}}]\label{auslander-reiten}
Let $\Lambda$ be an Artin algebra. Then there is a one-to-one correspondence between isomorphism classes of basic cotilting modules and contravariantly finite resolving subcategories of finite resolution dimension. This correspondence is given by $T\mapsto {}^{\perp_{>0}} T=\{X\in \modu{\Lambda}\mid \Ext^i_\Lambda(X,T)=0\text{ for all }i>0\}$. 
\end{thm}

Finiteness of relative global dimension in the setup of a regular exact Borel subalgebra of a right standardly stratified algebra follows from two facts: Firstly, the $\mathcal{F}(\overline{\Delta})$-resolution dimension is finite. Secondly, for a regular exact Borel subalgebra, the category of induced modules coincides with the category of $\overline{\Delta}$-filtered modules. 

For quasi-hereditary algebras, the finiteness of resolution dimension is obvious since the algebra itself is of finite global dimension. However, left and right standardly stratified algebras are not of finite global dimension, unless they are quasi-hereditary. Therefore, it is not as obvious that the resolution dimension is finite. It is however known even in this case. 

\begin{lem}\label{finite_resolution_dimension} 
Let $R$ be a right standardly stratified algebra.  Then the $\mathcal{F}(\overline{\Delta})$-resolution dimension is always finite. 
\end{lem}

\begin{proof}
For right standardly stratified algebras, the category $\mathcal{F}(\overline{\Delta})$ is resolving. Furthermore, according to \cite[Theorem 2.6 (vi)]{AHLU00} it always equals ${}^{\perp_{>0}}(T')$ where $T'$ is the characteristic cotilting module, which exists by the dual of \cite[Theorem 2.1]{AHLU00}. Using Theorem \ref{auslander-reiten}, this means that the $\mathcal{F}(\overline{\Delta})$-resolution dimension is finite.  
\end{proof}

As a remaining step to finish the proof of $A=R$ in \eqref{mainthm3:c}, it thus suffices to prove that for a right standardly stratified algebra $R$ with a regular exact Borel subalgebra $A$, the embedding induces an equivalence of categories $\Ind_A^R\cong \mathcal{F}(\overline{\Delta})$. This statement is mentioned in the case of quasi-hereditary algebras for a specifically constructed regular exact Borel subalgebra in \cite[Theorems 10.4 and 11.3]{KKO14}. For right standardly stratified algebras the analogous result is \cite[Theorem 6.15]{Got24}. However, the definition of a regular exact Borel subalgebra (or more generally a homological exact Borel subalgebra) readily yields the result as follows: Since $R\otimes_A L_\mathtt{i}\cong \overline{\Delta}_\mathtt{i}$ and $R\otimes_A -$ is an exact functor, it is immediate that the essential image $\Ind_A^R$ of the induction functor is contained in the category $\mathcal{F}(\overline{\Delta})$. The two categories agree if and only if $\Ind_A^R$ is closed under extensions. This however, follows from the surjectivity of the map $\Ext^1_A(M,N)\to \Ext^1_R(R\otimes_A M,R\otimes_A N)$ induced from the induction functor, which follows by induction from the isomorphism $\Ext^1_A(L_\mathtt{i}^A,L_\mathtt{j}^A)\cong \Ext^1_R(R\otimes_A L_\mathtt{i}^A, R\otimes_A L_\mathtt{j}^A)$. This finishes the proof of the $A=R$ part of  Theorem \ref{mainthm3} \eqref{mainthm3:c}.

We are left with proving that $R$ is semisimple (which as a consequence gives an alternative reason why $A=R$). Note that because $A\subseteq R$ is a Frobenius extension, there is a commutative diagram:
\[\begin{tikzcd}
  \Mod{R} \ar{rr}{\simeq}[swap]{F} && \Mod{L}  \\
  &  \arrow{ul}{R \otimes_A -} \Mod{A}. \arrow{ur}[swap]{\Hom_A(L,-)} & 
\end{tikzcd}\]
Because $A\subseteq R$ is an exact Borel subalgebra of the right standardly stratified algebra $R$, the induction functor $R\otimes_A-$ sends the simple $A$-module $L_\mathtt{i}^A$ to the proper standard module $\overline{\Delta}_\mathtt{i}^R$. We claim that $L$ is left standardly stratified and the coinduction functor $\Hom_A(L,-)$ sends $L_\mathtt{i}^A$ to the proper costandard module $\overline{\nabla}_\mathtt{i}^L$. This claim finishes the proof as the proper standard module $\overline{\Delta}_\mathtt{i}^R$ has the simple $R$-module $L_\mathtt{i}^R$ only once as a composition factor, namely at the top while the proper costandard module $\overline{\nabla}_\mathtt{i}^L$ has this as a composition factor only in its socle. Therefore, all proper standard modules as well as all proper costandard modules are simple, whence the algebra is semisimple. In the case of quasi-hereditary algebras over an algebraically closed field, to prove $\Hom_A(L,L_\mathtt{i}^A)\cong \overline{\nabla}_\mathtt{i}^L$ we proceed as follows: Firstly, we show that the algebra extension $A\subseteq R$ gives rise to a directed coring, this is \cite[Theorem 3.13]{BKK20}. Secondly, given a  directed coring, its left algebra is quasi-hereditary with (proper) costandard modules given by the coinduced modules $\Hom_A(L,L_\mathtt{i}^A)$, this is proved in \cite[Theorem 11.2]{KKO14}. 

For right standardly stratified algebras over a not-necessarily algebraically closed field, the proof strategy is the same, but needs to be adapted slightly. We start by introducing the appropriate notion of directedness. In the case of an algebraically closed field, this   is first mentioned in \cite[Definition 3.2]{Got24}:

\begin{defn}
A coring $\mathcal{C}=(A,V)$ with surjective counit $\varepsilon$, whose kernel is denoted by $\overline{V}$, is called \emphbf{one-cyclic directed} with respect to a partial order $\leq$ if the following criteria are satisfied:
\begin{itemize}
\item $\overline{V}$ is left projectivising, i.e.\ for all $A$-modules $M$, $\overline{V}\otimes_A M$ is projective as a left $A$-module. 
\item $A$ is one-cyclic directed, i.e. $\Ext^1_A(L_\mathtt{i}^A,L_\mathtt{j}^A)=0$ for $\mathtt{i}\nleq \mathtt{j}$.
\item $\Hom_A(\overline{V}\otimes_A L_\mathtt{i}^A,L_\mathtt{j}^A)=0$ for $\mathtt{i} \nless   \mathtt{j}$. 
\end{itemize}
\end{defn}

\begin{rmk}
In case the ground field is algebraically closed, by \cite[Theorem 3.1]{AR91b}, a left projectivising and right projective bimodule is automatically a projective bimodule and decomposes into a direct sum of indecomposable modules of the form $Ae_\mathtt{j}\otimes_\Bbbk e_\mathtt{i}A$. Therefore, the last condition can equivalently be rephrased as these direct summands satisfy $\mathtt{i}<\mathtt{j}$. So we see that in case of an algebraically closed field, our notion of one-cylic directed coincides with the one in \cite{Got24}. 
\end{rmk}

\begin{prop}
Let $A\subseteq R$ be a regular exact Borel subalgebra of a right standardly stratified algebra $R$. Then its dual coring $\mathcal{C}=(A,V)$ where $V=\Hom_A(R_A,A_A)$ is one-cyclic directed. 
\end{prop}

\begin{proof}
First we prove that the counit of $\mathcal{C}$ is surjective. For this we use \cite[Theorem 2.5 (i)]{BKK20} which shows that for this it suffices to prove that $R$ is a projective generator as a right $A$-module.  As $R$ is right standardly stratified, its dual $DR$ admits a filtration by costandard $R$-modules with each costandard module appearing at least once. By \cite[Proposition 4.5]{CK26}, each costandard $R$-module restricts to the corresponding injective $A$-module. Thus, the filtration splits as $A$-modules and $DR$ is an injective cogenerator as a left $A$-module. By dualising, it follows that $R$ is a projective generator as a right $A$-module. 

Given that the counit of the dual coring $\mathcal{C}$ is surjective, there exists a short exact sequence 
\[
\begin{tikzcd}
0\arrow{r} &\overline{V}\arrow{r} &V\arrow{r}{\varepsilon} &A\arrow{r} &0,    
\end{tikzcd}
\]
which splits as a short exact sequence of right $A$-modules. Previous work, \cite[Theorem 2.13]{BKK20}, implies that if in addition $A\subseteq R$ is regular, the bimodule $\overline{V}$ is left projectivising. By the definition of a regular exact Borel subalgebra, the algebra $A$ is one-cyclic directed. It remains to prove that $\Hom_A(\overline{V}\otimes_A L_\mathtt{i}^A,L_\mathtt{j}^A)=0$ for $\mathtt{i}\nless  \mathtt{j}$. For this we apply $\Hom_A(-\otimes_A L_\mathtt{i}^A,L_\mathtt{j}^A)$ to the above short exact sequence. This gives rise to the following long exact sequence
\[
\hspace{-2.4cm}
\begin{tikzcd}[row sep = 1ex]
0\arrow{r} &\Hom_A(L_\mathtt{i}^A, L_\mathtt{j}^A)\arrow{r} &\Hom_R(\overline{\Delta}_\mathtt{i}^R,\overline{\Delta}_\mathtt{j}^R)\arrow{r} &\Hom_A(\overline{V}\otimes_A L_\mathtt{i}^A,L_\mathtt{j}^A) \arrow[out=-5,
    in=175, looseness=1]{lld}\\
&\Ext^1_A(L_\mathtt{i}^A,L_\mathtt{j}^A)\arrow{r} &\Ext^1_R(\overline{\Delta}_\mathtt{i}^R,\overline{\Delta}_\mathtt{j}^R)\arrow{r} &0.
\end{tikzcd}
\]
For $\mathtt{i}\nleq \mathtt{j}$, the terms surrounding $\Hom_A(\overline{V}\otimes_A L_\mathtt{i}^A,L_\mathtt{j}^A)$ vanish, and so does it as well. For $\mathtt{i}=\mathtt{j}$, the requirement $\End_A(L_\mathtt{i}^A)\cong \End_R(L_\mathtt{i}^R)$ yields that the first map in the sequence is an isomorphism. Regularity implies that the last map is an isomorphism, therefore the term $\Hom_A(\overline{V}\otimes_A L_\mathtt{i}^A,L_\mathtt{i}^A)$ vanishes. 
\end{proof}

We conclude by proving that the two dual algebras to a one-cyclic directed coring 
are right resp.~left standardly stratified. This is sketched in \cite{Got24} in the algebraically closed case. Our proof drops this assumption and clarifies the following subtle points readers of \cite{Got24} may worry about. Even in the quasi-hereditary case,  $R\otimes_A L_\mathtt{i}^A$ is in general not indecomposable. Moreover, one needs to verify that the induced modules $R\otimes_A L_\mathtt{i}^A$ coincide with the proper standard
modules for some partial ordering on the simple $R$-modules.

\begin{prop}\label{rightstandardlyofonecyclic}
Let $\mathcal{C}=(A,V)$ be a one-cyclic directed coring. Then its right algebra $R$ is right standardly stratified with proper standard modules $\overline{\Delta}_\mathtt{i}=R\otimes_A L_\mathtt{i}$.  Dually, its left algebra $L$ is left standardly stratified with proper costandard modules $\overline{\nabla}_\mathtt{i}=\Hom_A(L,L_\mathtt{i})$. 
\end{prop}

\begin{proof}
We start by considering the following sequence of isomorphisms: 
\begin{equation}\label{chainisomorphism}
\begin{aligned}
\Hom_R(R\otimes_A P_\mathtt{i}^A, R\otimes_A L_\mathtt{j}^A)&\cong\Hom_A(V\otimes_A P_\mathtt{i}^A, L_\mathtt{j}^A)\\
&\cong \Hom_A(P_\mathtt{i}^A\oplus (\overline{V}\otimes_A P_\mathtt{i}^A), L_\mathtt{j}^A)\\
&\cong \Hom_A(P_\mathtt{i}^A, L_\mathtt{j}^A)\oplus \Hom_A(\overline{V}\otimes_A P_\mathtt{i}^A, L_\mathtt{j}^A) 
\end{aligned}
\end{equation}
where the middle isomorphism follows from applying $-\otimes_A P_\mathtt{i}^A$ to the sequence of $A$-bimodules $0\to \overline{V}\to V\to A\to 0$ and using that $P_\mathtt{i}^A$ is projective. Note that in the last line, the first summand vanishes for $\mathtt{i}\neq \mathtt{j}$ and the second summand vanishes for $\mathtt{i}\nless \mathtt{j}$ according to the third property of a one-cyclic directed coring combined with the second one, i.e.
\begin{itemize} 
\item for $\mathtt{i}\nleq \mathtt{j}$, the $\Hom$-space in \eqref{chainisomorphism} vanishes (here we use that the second condition of one-cyclic directed implies that all composition factors of $P_\mathtt{i}^A$ are of the form $L_\mathtt{j}^A$ for $\mathtt{j}\geq \mathtt{i})$ and
\item for $\mathtt{i}=\mathtt{j}$ the $\Hom$-space in \eqref{chainisomorphism} is isomorphic to $\Hom_A(P_\mathtt{i}^A,L_\mathtt{i}^A)\cong \End_A(L_\mathtt{i}^A)$. 
\end{itemize}

Therefore, there is at most one indecomposable summand of $R\otimes_A P_\mathtt{i}^A$, with a non-zero map to $R\otimes_A L_\mathtt{i}^A$. Assume that there is a decomposition of $R$-modules $R\otimes_A P_\mathtt{i}^A\cong \bigoplus Q$. Then $\End_A(L_\mathtt{i}^A)\cong \Hom_R(R\otimes_A P_\mathtt{i}^A, R\otimes_A L_\mathtt{i}^A) \cong \bigoplus \Hom_R(Q, R\otimes_A L_\mathtt{i}^A)$ as left $\End_A(L_\mathtt{i}^A)$-modules. However, since $L_\mathtt{i}^A$ is simple, $\End_A(L_\mathtt{i}^A)$ is a skew field and only one of the summands can be non-zero. Denote this indecomposable summand by $P_\mathtt{i}^R$. We claim that these form a complete set of indecomposable projective $R$-modules. First note that they are all non-isomorphic since for $\mathtt{i}\nleq \mathtt{j}$, the space $\Hom_R(P_\mathtt{i}^R, R\otimes_A L_\mathtt{j}^A)$ vanishes while the space $\Hom_R(P_\mathtt{j}^R, R\otimes_A L_\mathtt{j}^A)$ does not. Furthermore, observe that since there is a surjective map $R\otimes_A P_\mathtt{i}^A\twoheadrightarrow  R\otimes_A L_\mathtt{i}^A$ and $P_\mathtt{i}^R$ is the only summand of $R\otimes_A P_\mathtt{i}^A$ mapping non-trivially to $R\otimes_A L_\mathtt{i}^A$, there is a surjective map $P_\mathtt{i}^R\twoheadrightarrow  R\otimes_A L_\mathtt{i}^A$. By the horseshoe lemma, the $P_\mathtt{i}^R$ generate every module filtered by $R\otimes_A L_\mathtt{i}^A$, in particular, $R\cong R\otimes_A A$. It follows that $R$ is a direct sum of modules isomorphic to $P_\mathtt{i}^R$. 

Transferring the partial order on the simple $A$-modules to the $P_i^R$ yields an induced partial ordering on the simple $R$-modules. Using the calculation \eqref{chainisomorphism} of $\Hom$-spaces, with respect to this ordering, all composition factors of $R\otimes_A L_\mathtt{i}^A$ are of the form $L_\mathtt{j}^R$ for $\mathtt{j}\leq \mathtt{i}$ with $\mathtt{i}$ occurring only once. It follows that $R\otimes_A L_\mathtt{i}^A$ is a quotient of the maximal such highest weight module $\overline{\Delta}_\mathtt{i}^R$.

To show that the $R\otimes_A L_\mathtt{i}^A$ are not proper quotients of the $\overline{\Delta}_\mathtt{i}^R$, it suffices to show that they cannot be extended by simple modules $L_\mathtt{j}^R$ for $\mathtt{j}<\mathtt{i}$. For this we show by induction on the partial order on the simples that $\Ext^l_R(R\otimes_A L_\mathtt{i}^A, L_\mathtt{j}^R)=0$ for all $l>0$ and $\mathtt{j}<\mathtt{i}$. For $\mathtt{j}$ minimal, the module $R\otimes_A L_\mathtt{j}^A$ only has one composition factor, and is therefore isomorphic to the simple module $L_\mathtt{j}^R$. It follows that 
\[\Ext^l_R(R\otimes_A L_\mathtt{i}^A, L_\mathtt{j}^R)\cong \Ext^l_R(R\otimes_A L_\mathtt{i}^A, R\otimes_A L_\mathtt{j}^A)=0,\]
since the latter is a quotient of $\Ext^l_A(L_\mathtt{i}^A,L_\mathtt{j}^A)$, which vanishes since $A$ is one-cyclic directed and $\mathtt{j}<\mathtt{i}$. Now assume by induction that $\Ext^l_R(R\otimes_A L_\mathtt{i}^A, L_\mathtt{k}^R)=0$ for all $l>0$ and $\mathtt{k}<\mathtt{j}<\mathtt{i}$. Then, as all composition factors of $\radoperator(R\otimes_A L_\mathtt{j}^A)$ satisfy $\mathtt{k}<\mathtt{j}$, a long exact sequence in cohomology argument shows that $\Ext^l_R(R\otimes_A L_\mathtt{i}^A, \radoperator(R\otimes_A L_\mathtt{j}^A))=0$ for all $l>0$. However, applying $\Hom_R(R\otimes_A L_\mathtt{i}^A,-)$
 to the short exact sequence 
\[
\begin{tikzcd}
0\arrow{r} &\radoperator(R\otimes_A L_\mathtt{j}^A) \arrow{r} &R\otimes_A L_\mathtt{j}^A\arrow{r} &L_\mathtt{j}^R\arrow{r} &0    
\end{tikzcd}
\]
yields 
\[
\begin{tikzcd}[column sep=2ex]
\Ext^l_R(R\otimes_A L_\mathtt{i}^A, R\otimes_A L_\mathtt{j}^A)\arrow{r} &\Ext^l_R(R\otimes_A L_\mathtt{i}^A, L_\mathtt{j}^R)
\arrow{r} &\Ext^{l+1}_R(R\otimes_A L_\mathtt{i}^A, \radoperator(R\otimes_A L_\mathtt{j}^A)),
\end{tikzcd}
\]
and the middle term vanishes since the outer terms vanish. Consequently, $R\otimes_A L_\mathtt{i}^A\cong \overline{\Delta}_\mathtt{i}^R$. Thus, we have shown that $R$ has a filtration by the proper standard modules $\overline{\Delta}_\mathtt{i}$ and is therefore right standardly stratified.  
\end{proof}

\begin{rmk}
The approach taken in this article to prove statement \eqref{mainthm3:c} in Theorem
\ref{mainthm3} for right standardly stratified algebras fails for left
standardly stratified algebras for the following reason:

Let $R$ be a left standardly stratified algebra. Then the  
$\mathcal{F}(\Delta)$-resolution dimension is finite if and only if $R$ is
quasi-hereditary.
  
To prove this claim, recall that
the category $\mathcal{F}(\Delta)$ is a resolving subcategory of
$\modu{R}$. Using Theorem \ref{auslander-reiten}, such a resolving subcategory
has finite resolution dimension if and only if it can be written as 
\[{}^{\perp_{>0}} T=\{X\in \modu{R}\mid \Ext^i_R(X,T)=0\text{ for all }i>0\}\] 
for a cotilting module $T$. However, according to \cite[Theorem 2.4]{AHLU00}
this holds if and only if $R$ is quasi-hereditary.   
\medskip

Moreover, a result for left standardly stratified algebras necessarily has
to be different from Theorem \ref{mainthm3}, for the following reason:

Let $R$ be a left standardly stratified algebra. Assume that $R$ is its own regular exact Borel subalgebra. Then $R$ is also directed, so in particular quasi-hereditary. On the other hand, it can happen that a left standardly stratified algebra $R$ has a regular exact Borel subalgebra $A$ such that $A\subseteq R$ is a Frobenius extension. For example consider a local Frobenius algebra $R$. Then the only standard module is $R$, thus a regular exact Borel subalgebra is given by the semisimple part of $R$, which is $A$, and $A\subseteq R$ is a Frobenius extension since $R$ is a Frobenius algebra. 
\end{rmk}

\section*{Acknowledgments}
Tomasz Brzezi\'nski would like to express his gratitude for the hospitality of Institute of Algebra and Number Theory at University of Stuttgart, where the work on this paper has started. The research  of Tomasz Brzezi\'nski is partially supported by the National Science Centre, Poland, through the WEAVE-UNISONO grant no.\ 2023/05/Y/ST1/00046. Teresa Conde acknowledges support by the Deutsche Forschungsgemeinschaft (DFG, German Research Foundation) -- Project-ID 491392403 - TRR 358. Julian K\"ulshammer acknowledges support by the Swedish Research Council -- registration
number 2024-05357. The authors want to thank Anna Rodriguez Rasmussen for help with providing a cleaner argument for Proposition \ref{rightstandardlyofonecyclic}.

\bibliographystyle{alpha}
\bibliography{publication}

\end{document}